\documentclass[12pt, reqno]{amsart}
\usepackage{amsmath, amstext, amsbsy, amssymb, amscd}
\usepackage[mathscr]{eucal}

\usepackage{tikz}
\usepackage{verbatim}
\usetikzlibrary{arrows,shapes}

\makeatletter
\@namedef{subjclassname@2020}{%
  \textup{2020} Mathematics Subject Classification}
\makeatother

\newtheorem{theorem}{Theorem}[section]
\newtheorem{lemma}[theorem]{Lemma}
\newtheorem{proposition}[theorem]{Proposition}
\newtheorem{corollary}[theorem]{Corollary}
\theoremstyle{definition}
\newtheorem{definition}[theorem]{Definition}
\newtheorem{example}[theorem]{Example}

\theoremstyle{remark}
\newtheorem{remark}[theorem]{Remark}
\theoremstyle{problem}

\newtheorem{conjecture}[theorem]{Conjecture}

\numberwithin{equation}{section}

\newcommand{\ch}{{\rm ch} }

\newcommand{\Coe}{ {\rm Coeff} }
\newcommand{\C}{ \mathbb C }

\newcommand{\End}{{\rm End}}

\newcommand{\fa}{ \mathfrak a}

\newcommand{\fG}{ \mathfrak G }
\newcommand{\fL}{ \mathfrak L }
\newcommand{\fn}{ \mathfrak n }
\newcommand{\fock}{{\mathbb H}_X}

\newcommand{\Hn}{H^*(\Xn)}

\newcommand{\la}{\lambda}
\newcommand{\lambsq}{s(\lambda)}
\newcommand{\Ln}{L^{[n]}}

\newcommand{\MD}{\mathcal {MD}}

\newcommand{\N}{\mathbb N}

\newcommand{\Q}{{\mathbb Q}}

\newcommand{\qMD}{{\rm q}{\mathcal {MD}}}
\newcommand{\qMZV}{ {\bf qMZV} }

\newcommand{\Tr}{ {\rm Tr} }
\newcommand{\vac}{|0\rangle}

\newcommand{\w}{\tilde}

\newcommand{\Wb}{ {\bf W} }

\newcommand{\Xn}{ {X^{[n]}}}
\newcommand{\Z}{ \mathbb Z }

\def\beq{\begin{equation}}
\def\eeq{\end{equation}}

\begin{document}

\title[Hilbert schemes of points and multiple $q$-zeta values]
      {Hilbert schemes of points on surfaces
      and multiple $q$-zeta values}

\author[Mazen M. Alhwaimel]{Mazen M. Alhwaimel}
\address{Department of Mathematics, College of Science, Qassim University, P. O. Box 6644, Buraydah 51452, Saudi Arabia} 
\email{alhwaimelm@gmail.com}

\author[Zhenbo Qin]{Zhenbo Qin}
\address{Department of Mathematics, University of Missouri, 
         Columbia, MO 65211, USA} 
\email{qinz@missouri.edu}

\date{\today}
\keywords{Hilbert schemes of points on surfaces; 
multiple $q$-zeta values; quasimodular forms; Heisenberg operators; 
generalized partitions.} 
\subjclass[2020]{Primary 14C05; Secondary 11M32}

\begin{abstract}
For a line bundle $L$ on a smooth projective surface $X$
and nonnegative integers $k_1, \ldots, k_N$, 
Okounkov \cite{Oko} introduced the reduced generating series 
$\big \langle \ch_{k_1}^{L} \cdots \ch_{k_N}^{L} \big \rangle'$ 
for the intersection numbers among the Chern characters of 
the tautological bundles over the Hilbert schemes of points on $X$ 
and the total Chern classes of the tangent bundles of 
these Hilbert schemes, and conjectured that they are 
multiple $q$-zeta values of weight at most $\sum_{i=1}^N (k_i + 2)$. 
The second-named author further conjectured in \cite{Qin2} that 
these reduced generating series are quasi-modular forms if 
the canonical divisor of $X$ is numerically trivial. 
In this paper, we verify these two conjectures for 
$\big \langle \ch_2^{L} \big \rangle'$. 
The main approaches are to apply the procedure laid out in \cite{QY} and 
to establish various identities for multiple $q$-zeta values 
and quasi-modular forms.
\end{abstract}
 
\maketitle

\section{\bf Introduction} 
\label{sect_intr}

Hilbert schemes 
have been studied extensively since the pioneering work of 
Grothendieck \cite{Grot}. It is well known \cite{Bri, Fog, Iar} 
that the Hilbert schemes of points, 
parametrizing $0$-dimensional closed subschemes,
on a smooth projective surface are smooth and irreducible. 
Let $X$ be a smooth projective complex surface, 
and let $\Xn$ be the Hilbert scheme of $n$ points in $X$. A line bundle 
$L$ on $X$ induces a tautological rank-$n$ bundle $\Ln$ on $\Xn$.
Let $\ch_k(\Ln)$ be the $k$-th Chern character of $\Ln$.
Following Okounkov \cite{Oko}, we introduce the two generating series:
\begin{eqnarray}     
\big \langle \ch_{k_1}^{L_1} \cdots \ch_{k_N}^{L_N} \big \rangle
  &=&\sum_{n \ge 0} q^n \, \int_\Xn \ch_{k_1}(L_1^{[n]}) \cdots \ch_{k_N}(L_N^{[n]}) 
        \cdot c(T_\Xn)   \label{OkoChkN.1}   \\
\big \langle \ch_{k_1}^{L_1} \cdots \ch_{k_N}^{L_N} \big \rangle'
  &=&\frac{\big \langle \ch_{k_1}^{L_1} \cdots \ch_{k_N}^{L_N} 
     \big \rangle}{\langle \rangle}
      = (q; q)_\infty^{\chi(X)} \cdot \big \langle 
      \ch_{k_1}^{L_1} \cdots \ch_{k_N}^{L_N} \big \rangle  \label{OkoChkN.2} 
\end{eqnarray}
where $0 < |q| < 1$, $c\big (T_\Xn \big )$ is the total Chern class of 
the tangent bundle $T_\Xn$, $\chi(X)$ is the Euler characteristics of $X$, 
$(a; q)_n = \prod_{i=0}^n (1-aq^i)$, and  
$$
  \langle \rangle 
= \sum_{n \ge 0} q^n \, \int_\Xn c(T_\Xn)
= \frac1{(q; q)_\infty^{\chi(X)}}
$$
which is a formula due to G\"ottsche \cite{Got}.
In \cite{Car1, Car2, Zhou}, 
for $X = \C^2$ with a suitable $\C^*$-action and $L = \mathcal O_X$,
the series $\big \langle \ch_{k_1}^L \cdots \ch_{k_N}^L \big \rangle$ 
in the equivariant setting has been studied.
In \cite{Oko, Qin2}, the following conjectures were proposed.

\begin{conjecture}   \label{OkoConj}
(\cite[Conjecture~2]{Oko}) Let $L$ be a line bundle on a smooth 
projective surface $X$. Then 
$\big \langle \ch_{k_1}^L \cdots \ch_{k_N}^L \big \rangle'$ is 
a multiple $q$-zeta value of weight at most 
$$
\sum_{i=1}^N (k_i + 2).
$$
\end{conjecture}

\begin{conjecture}   \label{QinConj}
(\cite{Qin2}) Let $L$ be a line bundle on a smooth 
projective surface $X$. If the canonical divisor of $X$ 
is numerically trivial, then 
$\big \langle \ch_{k_1}^L \cdots \ch_{k_N}^L \big \rangle'$ is  
a quasimodular form of weight at most $\sum_{i=1}^N (k_i + 2)$.
\end{conjecture}

In the region ${\rm Re} \, s > 1$, the Riemann zeta function is defined by 
$$
\zeta(s) = \sum_{n =1}^\infty \frac{1}{n^{s}}.
$$
The integers $s > 1$ give rise to a sequence of special values of the Riemann zeta function.
Multiple zeta values are series of the form
$$
\zeta(s_1, \ldots, s_k) = \sum_{n_1 > \cdots > n_k > 0} 
\frac{1}{n_1^{s_1} \cdots n_k^{s_k}}
$$
where $s_1, \ldots, s_k$ are positive integers with $s_1 > 1$, 
and $n_1, \ldots, n_k$ denote positive integers.
Multiple $q$-zeta values are $q$-deformations of 
$\zeta(s_1, \ldots, s_k)$, which may take different forms 
\cite{Bac, BK2, Bra1, Bra2, OT, Zhao, Zud}. 
The multiple $q$-zeta values defined by Okounkov \cite{Oko} 
are denoted by $Z(s_1, \ldots, s_k)$ where $s_1, \ldots, s_k > 1$ 
are integers (see Definition~\ref{def_qMZV}~(iv)). 
For instance, we have
$$
Z(2) = \sum_{n>0} \frac{q^n}{(1-q^n)^2}, \quad
Z(3) = \sum_{n>0} \frac{q^n(q^n+1)}{(1-q^n)^3}, 
$$
\begin{eqnarray}   \label{exampleZs}
Z(4) = \sum_{n>0} \frac{q^{2n}}{(1-q^n)^4}, \quad
Z(6) = \sum_{n>0} \frac{q^{3n}}{(1-q^n)^6}.
\end{eqnarray}
The weight of $Z(s_1, \ldots, s_k)$ is defined to be 
$s_1 + \ldots + s_k$. 
By \cite[Theorem~2.4]{BK3}, the $\Q$-linear span $\qMZV$ of 
all the multiple $q$-zeta values $Z(s_1, \ldots, s_k)$ with 
$s_1, \ldots, s_k > 1$ is an algebra over $\Q$. 
By \eqref{20170812522pm}, the set ${\bf QM}$ of all quasi-modular forms 
(of level $1$ on the full modular group ${\rm PSL}(2; \Z)$) over $\Q$ 
is a subalgebra of $\qMZV$:
$$
{\bf QM} = \Q\big [Z(2), Z(4), Z(6) \big ] \subset \qMZV.
$$

Conjecture~\ref{OkoConj} asserts that 
$\big \langle \ch_{k_1}^L \cdots \ch_{k_N}^L \big \rangle' \in \qMZV$,
and holds for $\big \langle \ch_1^L \big \rangle'$:
\begin{eqnarray}   \label{int_203001031038uuu.1}
  \big \langle \ch_{1}^L \big \rangle'
= \frac{1}{2} \big (Z(2) - Z(3) \big ) \cdot K_X^2 - Z(2) \cdot K_X L.
\end{eqnarray}
(see \eqref{203001031038uuu.1} below and \cite[Corollary~3]{CO}). 
Conjecture~\ref{QinConj} has been confirmed 
for $\big \langle \ch_{k_1}^L \cdots \ch_{k_N}^L \big \rangle'$ 
with $X = \C^2$ and $L = \mathcal O_X$ in the equivariant setting of 
\cite{Car1, Car2} where the equivariant canonical divisor of 
$X = \C^2$ is indeed trivial. By \eqref{int_203001031038uuu.1}, 
Conjecture~\ref{QinConj} is trivially true for 
$\big \langle \ch_1^L \big \rangle'$.

The main result of our paper is the following theorem about 
$\langle \ch_2^L \rangle'$.

\begin{theorem} \label{Intro_theorem_ch2L}
Let $L$ be a line bundle over a smooth projective surface $X$, 
and $K_X$ be the canonical divisor of $X$. Then, 
the reduced series $\langle \ch_2^L \rangle'$ is equal to
$$ 
\left (-\frac{7}{24} Z(4) - \frac{23}{24} Z(2)^2 \right ) \chi(X)
+ \frac12 (Z(3) - Z(2)) \cdot K_X L
$$
$$ 
+ \left (\frac{13}{12} Z(4) - \frac13 Z(2)^2 - \frac14 Z(3) 
+ \frac14 Z(2) \right ) K_X^2 + \frac12 Z(2) \cdot L^2.
$$
In particular, Conjecture~\ref{OkoConj} holds for the reduced series 
$\langle \ch_2^L \rangle'$.
\end{theorem}

Since ${\bf QM} = \Q\big [Z(2), Z(4), Z(6) \big ]$, 
we immediately obtain the following.

\begin{corollary} \label{Intro_corollary_ch2L}
Let $L$ be a line bundle over a smooth projective surface $X$. 
If the canonical divisor of $X$ is numerically trivial, 
then $\langle \ch_2^L \rangle'$ 
is a quasi-modular form of weight at most $4$. 
In particular, Conjecture~\ref{QinConj} holds for 
$\langle \ch_2^L \rangle'$.
\end{corollary}

To prove Theorem~\ref{Intro_theorem_ch2L} (= Theorem~\ref{theorem_ch2L}), 
we first reduce the generating series $\langle \ch_2^L \rangle'$ to 
$(q;q)_\infty^{\chi(X)} \cdot F_2^\alpha(q)$ where 
$$
  F^{\alpha_1, \ldots, \alpha_N}_{k_1, \ldots, k_N}(q) 
= \sum_{n \ge 0} q^n \int_\Xn \left ( \prod_{i=1}^N 
  G_{k_i}(\alpha_i, n) \right ) c\big (T_\Xn \big )
$$
for cohomology classes $\alpha_1, \ldots, \alpha_N \in H^*(X)$ 
and nonnegative integers $k_1, \ldots, k_N$, 
and $G_{k}(\alpha, n)$ is the homogeneous component in 
$H^{|\alpha|+2k}(\Xn)$ of \eqref{DefOfGGammaN} 
for a homogeneous class $\alpha \in H^*(X)$. 
By Lemma~\ref{FtoW},  
$F^{\alpha_1, \ldots, \alpha_N}_{k_1, \ldots, k_N}(q)$ can be 
expressed in terms of the Ext vertex operators of 
Carlsson and Okounkov \cite{Car1, Car2, CO} and 
the Chern character operators $\mathfrak G_{k_i}(\alpha_i)$ 
which act on $H^*(\Xn)$ by the cup product with the classes
$G_{k_i}(\alpha_i, n)$. A formula for $\mathfrak G_2(\alpha)$ 
is given in \cite{Tang}, even though there is no formula for 
$\mathfrak G_k(\alpha)$ for a general $k > 2$.
Utilizing the procedure in \cite{QY},
we can compute $(q;q)_\infty^{\chi(X)} \cdot F_2^\alpha(q)$ 
explicitly as a power series of $q$. Let $1_X$ and $e_X$ be 
the fundamental class and Euler class of the surface $X$ respectively. 
After establishing various identities for multiple $q$-zeta values 
and quasi-modular forms, 
we will express $(q;q)_\infty^{\chi(X)} \cdot F_2^\alpha(q)$ as 
a linear combination of the intersection pairings
$$
\langle e_X, \alpha \rangle, \langle 1_X, \alpha \rangle, 
\langle K_X, \alpha \rangle, \langle K_X^2, \alpha \rangle
$$ 
whose coefficients involving only 
$Z(2), Z(3), Z(4)$ (see Theorem~\ref{lemma_ch2L} for details). 

The paper is organized as follows. 
Section~\ref{sect_qMZV} is devoted to the preliminaries about 
multiple $q$-zeta values and quasi-modular forms. 
Section~\ref{sect_Hilbert} recalls the basics of 
the Hilbert schemes of points on a smooth projective surface,
the Heisenberg operators of Grojnowski and Nakajima, 
and the Chern character operators. In Section~\ref{sect_CO}, 
we introduce the generating series 
$F^{\alpha_1, \ldots, \alpha_N}_{k_1, \ldots, k_N}(q)$ 
in terms of the Chern character operators.
Section~\ref{sect_F2Alpha} calculates the reduced series 
$(q;q)_\infty^{\chi(X)} \cdot F_2^{\alpha}(q)$. 
Theorem~\ref{Intro_theorem_ch2L} (= Theorem~\ref{theorem_ch2L}) 
is verified in Section~\ref{sect_Application}.

\bigskip\noindent
{\bf Convention.} All cohomology groups in this paper have 
coefficients in $\C$.


\section{\bf Multiple $q$-zeta values} 
\label{sect_qMZV}

In this section, we will recall basic definitions and facts 
regarding multiple $q$-zeta values and quasi-modular forms.

The following definitions and notations are from \cite{BK1, BK3, Oko}.

\begin{definition}  \label{def_qMZV}
Let $\N = \{1,2,3, \ldots\}$, and fix a subset $S \subset \N$. 
\begin{enumerate}
\item[{\rm (i)}]
Let $Q = \{Q_s(t)\}_{s \in S}$ where each $Q_s(t) \in \Q[t]$ is a polynomial 
with $Q_s(0) = 0$ and $Q_s(1) \ne 0$. For $s_1, \ldots, s_\ell \in S$ 
with $\ell \ge 1$, define
$$
Z_Q(s_1, \ldots, s_\ell)
= \sum_{n_1 > \cdots > n_\ell \ge 1} \prod_{i=1}^\ell 
  \frac{Q_{s_i}(q^{n_i})}{(1-q^{n_i})^{s_i}}
\in \Q[[q]].
$$
Put $Z_Q(\emptyset) = 1$, and define $Z(Q, S)$ to be the $\Q$-linear span 
of the set
$$
\{Z_Q(s_1, \ldots, s_\ell)| \, \ell \ge 0 \text{ and }
s_1, \ldots, s_\ell \in S \}.
$$

\item[{\rm (ii)}]
Define $\mathcal {MD} = Z(Q^E, \N)$ where $Q^E = \{Q_s^E(t)\}_{s \in \N}$
and 
$$
Q_s^E(t) = \frac{tP_{s-1}(t)}{(s-1)!}
$$ 
with $P_s(t)$ being the Eulerian polynomial defined by 
\begin{eqnarray}   \label{def_qMZV.01}
\frac{tP_{s-1}(t)}{(1-t)^s} = \sum_{d=1}^\infty d^{s-1}t^d.
\end{eqnarray}  
We have $t P_{0}(t) = t$.
For $s > 1$, the polynomial $t P_{s-1}(t)$ has degree $s - 1$. 
For $s_1, \ldots, s_\ell \in \N$ with $\ell \ge 1$, 
define the Bachmann-K\" uhn series
\begin{eqnarray}   \label{def_qMZV.02}
[s_1, \ldots, s_\ell] 
= Z_{Q^E}(s_1, \ldots, s_\ell)
= \sum_{n_1 > \cdots > n_\ell \ge 1} \prod_{i=1}^\ell 
  \frac{Q_{s_i}^E(q^{n_i})}{(1-q^{n_i})^{s_i}}.
\end{eqnarray}  

\item[{\rm (iii)}]
Define $\text{q}\mathcal {MD}$ be the subspace of $\mathcal {MD}$ linearly 
spanned by $1$ and all the brackets $[s_1, \ldots, s_\ell]$ with $s_1 > 1$. 

\item[{\rm (iv)}]
Define $\qMZV = Z(Q^O, \N_{>1})$ where 
$Q^O = \{Q_s^O(t)\}_{s \in \N_{>1}}$ and 
$$
  Q_s^O(t) 
= \begin{cases}
  t^{s/2}            &\text{if $s \ge 2$ is even;} \\
  t^{(s-1)/2} (t+1)  &\text{if $s \ge 3$ is odd.}
  \end{cases}
$$ 
For $s_1, \ldots, s_\ell \in \N$ with $\ell \ge 1$, 
define the Okounkov series
$$
Z(s_1, \ldots, s_\ell) = Z_{Q^O}(s_1, \ldots, s_\ell).
$$
\end{enumerate}
\end{definition}

We see from Definition~\ref{def_qMZV}~(ii) that for $s \ge 1$, 
\begin{eqnarray}   \label{def_qMZV.03}
(s-1)! \cdot \frac{Q_{s}^E(t)}{(1-t)^s} = \sum_{d=1}^\infty d^{s-1}t^d.
\end{eqnarray} 
It follows that 
\begin{eqnarray*}   
  (s-1)! \cdot [s] 
= (s-1)! \cdot \sum_{n \ge 1} \frac{Q_{s}^E(q^n)}{(1-q^n)^s} 
= \sum_{n \ge 1} \sum_{d=1}^\infty d^{s-1} q^{nd}
\end{eqnarray*} 
for $s \ge 1$. Therefore, we obtain
\begin{eqnarray}   \label{def_qMZV.04}
  [s] 
= \frac{1}{(s-1)!} \sum_{n, d \ge 1} d^{s-1} q^{nd}
= \frac{1}{(s-1)!} \sum_{d \ge 1} d^{s-1} \frac{q^{d}}{1-q^d}.
\end{eqnarray}  
More generally, for $s_1, \ldots, s_\ell \ge 1$, we have
\begin{eqnarray}   \label{def_qMZV.05}
  [s_1, \ldots, s_\ell] 
= \frac{1}{(s_1-1)! \cdots (s_\ell-1)!} 
  \sum_{\substack{n_1 > \cdots > n_\ell \ge 1\\d_1, \ldots, d_\ell \ge 1}} 
  d_1^{s_1-1} \cdots d_\ell^{s_\ell-1} q^{n_1d_1+\ldots+n_\ell d_\ell}.
\end{eqnarray} 

Some examples of the Bachmann-K\" uhn series $[s], s \ge 1$ are
\begin{eqnarray}   \label{example[s]}
[1] = \sum_{n>0} \frac{q^n}{1-q^n}, \quad
[2] = \sum_{n>0} \frac{nq^n}{1-q^n}, \quad
[3] = \frac12 \sum_{n>0} \frac{n^2 q^n}{1-q^n}.
\end{eqnarray}
Some examples of the Okounkov series $Z(s), s \ge 1$ are given by 
\eqref{exampleZs}.

By the Theorem~2.13 and Theorem~2.14 in \cite{BK1},
$\text{q}\mathcal {MD}$ is a subalgebra of $\mathcal {MD}$, 
and $\mathcal {MD}$ is a polynomial ring over $\text{q}\mathcal {MD}$ 
with indeterminate $[1]$:
\begin{eqnarray}   \label{BK1Thm214}
\mathcal {MD} = \text{q}\mathcal {MD}[\,[1]\,].
\end{eqnarray}
By the Proposition~2.2 and Theorem~2.4 in \cite{BK3}, 
$Z(\{Q_s^E(t)\}_{s \in \N_{>1}}, \N_{>1})$ is a subalgebra of 
$\mathcal {MD}$ as well and 
$\qMZV = Z(\{Q_s^E(t)\}_{s \in \N_{>1}}, \N_{>1})$. Therefore,   
\begin{eqnarray}   \label{BK3-2.4}
\qMZV = Z(\{Q_s^E(t)\}_{s \in \N_{>1}}, \N_{>1}) \subset \qMD \subset \MD
\end{eqnarray} 
are inclusions of $\Q$-algebras.  
For instance, by \cite[Example~2.6]{BK3}, 
\begin{eqnarray}   \label{BK3-2.6}
Z(2) = [2], \quad Z(3) = 2[3], \quad Z(4) = [4] - \frac16 [2].
\end{eqnarray}

The graded ring ${\bf QM}$ of quasi-modular forms (of level $1$  
on the full modular group ${\rm PSL}(2; \Z)$) over $\Q$ is 
the polynomial ring over $\Q$ generated by the Eisenstein series 
$G_2(q), G_4(q)$ and $G_6(q)$:
$$
{\bf QM} = \Q[G_2, G_4, G_6] = {\bf M}[G_2]
$$
where ${\bf M} = \Q[G_4, G_6]$ is the graded ring of modular forms 
(of level $1$) over $\Q$, and 
\begin{eqnarray*}    
G_{2k} = G_{2k}(q) =\frac{1}{(2k-1)!} \cdot 
\left (-\frac{B_{2k}}{4k} + \sum_{n \ge 1} 
\Big ( \sum_{d|n} d^{2k-1} \Big )q^n \right )
\end{eqnarray*}
where $B_i \in \Q, i \ge 2$ are the Bernoulli numbers defined by 
$$
\frac{t}{e^t - 1} = 1 - \frac{t}{2} + \sum_{i =2}^{+\infty} B_i \cdot \frac{t^i}{i!}.
$$
The grading is to assign $G_2, G_4, G_6$ weights $2, 4, 6$ respectively.
By \cite[p.6]{BK3}, 
\begin{eqnarray*}
G_2 &=& -\frac{1}{24} + Z(2),      \\
G_4 &=& \frac{1}{1440} + Z(2) + \frac{1}{6}Z(4),     \\
G_6 &=& -\frac{1}{60480} + \frac{1}{120} Z(2) + \frac{1}{4} Z(4) + Z(6).
\end{eqnarray*}
It follows that
\begin{eqnarray}  \label{20170812522pm} 
{\bf QM} = \Q[Z(2), Z(4), Z(6)].
\end{eqnarray}

\section{\bf Hilbert schemes of points on surfaces} 
\label{sect_Hilbert}

In this section, we will review the Hilbert schemes of points 
on surfaces and the geometric construction of 
the Heisenberg operators by Grojnowski and Nakajima. 
Furthermore, we will recall the concepts of generalized partitions 
and the Chern character operators. 

Let $X$ be a smooth projective complex surface,
and $\Xn$ be the Hilbert scheme of $n$ points in $X$. 
An element in $\Xn$ is represented by a
length-$n$ $0$-dimensional closed subscheme $\xi$ of $X$. For $\xi
\in \Xn$, let $I_{\xi}$ be the corresponding sheaf of ideals. It
is well known that $\Xn$ is irreducible and smooth with dimension $2n$. 
Define the universal codimension-$2$ subscheme:
\begin{eqnarray*}
{\mathcal Z}_n=\{(\xi, x) \subset \Xn\times X \, | \, x\in
 {\rm Supp}{(\xi)}\}\subset \Xn\times X.
\end{eqnarray*}
Denote by $p_1$ and $p_2$ the projections of $\Xn \times X$ to
$\Xn$ and $X$ respectively. A line bundle $L$ on $X$ induces 
the tautological rank-$n$ bundle $\Ln$ over $\Xn$:
\begin{eqnarray}    \label{L[n]}
\Ln = (p_{1}|_{\mathcal Z_n})_*\big (p_2^*L|_{\mathcal Z_n} \big ).
\end{eqnarray}
Let $\fock$ be the direct sum of the cohomology groups of the Hilbert schemes $\Xn$:
\begin{eqnarray*}
\fock = \bigoplus_{n=0}^\infty \Hn.
\end{eqnarray*}
For $m \ge 0$ and $n > 0$, let $Q^{[m,m]} = \emptyset$ and define
$Q^{[m+n,m]}$ to be the closed subset:
$$
\{ (\xi, x, \eta) \in X^{[m+n]} \times X \times X^{[m]} \, | \,
\xi \supset \eta \text{ and } \mbox{Supp}(I_\eta/I_\xi) = \{ x \}\}.
$$

We recall from \cite{Gro, Nak} the geometric construction of 
the Heisenberg operators on the space $\fock$. 
Let $\alpha \in H^*(X)$. Set $\mathfrak a_0(\alpha) =0$.
For an integer $n > 0$, the operator $\mathfrak
a_{-n}(\alpha) \in \End(\fock)$ is
defined by
$$
\mathfrak a_{-n}(\alpha)(a) = \w{p}_{1*}([Q^{[m+n,m]}] \cdot
\w{\rho}^*\alpha \cdot \w{p}_2^*a)
$$
for $a \in H^*(X^{[m]})$, where $\w{p}_1, \w{\rho},
\w{p}_2$ are the projections of $X^{[m+n]} \times X \times
X^{[m]}$ to $X^{[m+n]}, X, X^{[m]}$ respectively. Define
$\mathfrak a_{n}(\alpha) \in \End(\fock)$ to be $(-1)^n$ times the
operator obtained from the definition of $\mathfrak
a_{-n}(\alpha)$ by switching the roles of $\w{p}_1$ and $\w{p}_2$. 
We often refer to $\mathfrak a_{-n}(\alpha)$ (resp. $\mathfrak a_n(\alpha)$) 
as the {\em creation} (resp. {\em annihilation})~operator. 
The following is from \cite{Gro, Nak}. Our convention of the sign follows \cite{LQW2}.

\begin{theorem} \label{commutator}
The operators $\mathfrak a_n(\alpha)$ satisfy
the commutation relation:
\begin{eqnarray*}
\displaystyle{[\mathfrak a_m(\alpha), \mathfrak a_n(\beta)] = -m
\; \delta_{m,-n} \cdot \langle \alpha, \beta \rangle \cdot {\rm Id}_{\fock}}.
\end{eqnarray*}
The space $\fock$ is an irreducible module over the Heisenberg
algebra generated by the operators $\mathfrak a_n(\alpha)$ with a
highest~weight~vector $\vac=1 \in H^0(X^{[0]}) \cong \C$.
\end{theorem}

In the above theorem, the Lie bracket is understood in the super
sense according to the parity of the degrees of the
cohomology classes involved, and 
$$
\langle \alpha, \beta \rangle = \int_X \alpha \beta.
$$ 
When $\alpha$ and $\beta$ are divisors on $X$, we also use 
$\alpha \beta$ to stand for $\langle \alpha, \beta \rangle$.

It follows from
Theorem~\ref{commutator} that the space $\fock$ is linearly
spanned by all the Heisenberg monomials $\mathfrak
a_{n_1}(\alpha_1) \cdots \mathfrak a_{n_k}(\alpha_k) \vac$
where $k \ge 0$ and $n_1, \ldots, n_k < 0$.

\begin{definition} \label{partition}
\begin{enumerate}
\item[{\rm (i)}]
A {\it generalized partition} \index{partition, generalized} of an integer $n$ is of the form
$$
\lambda = (\cdots  (-2)^{m_{-2}}(-1)^{m_{-1}} 1^{m_1}2^{m_2} \cdots)
$$ 
such that part $i\in \Z$ has multiplicity $m_i$, $m_i \ne 0$ for 
only finitely many $i$'s, and $n = \sum_i i m_i$. Define 
$\ell(\la) = \sum_i m_i$, $|\la| = \sum_i im_i$, 
$|\la|_+ = \sum_{i > 0} im_i$, $s(\la) = \sum_i i^2m_i$,
and $\lambda^! = \prod_i m_i!$.

\item[{\rm (ii)}] 
Let $\alpha \in H^*(X)$ and $k \ge 1$. Define $\tau_{k*}: H^*(X) \to H^*(X^k)$ 
to be the linear map induced by the diagonal embedding $\tau_k: X \to X^k$, and
$$
(\mathfrak a_{m_1} \cdots \mathfrak a_{m_k})(\alpha)
= \mathfrak a_{m_1} \cdots \mathfrak a_{m_k}(\tau_{k*}\alpha)
= \sum_j \mathfrak a_{m_1}(\alpha_{j,1}) \cdots \mathfrak a_{m_k}(\alpha_{j,k})
$$ 
when $\tau_{k*}\alpha = \sum_j \alpha_{j,1} \otimes \cdots 
\otimes \alpha_{j, k}$ via the K\"unneth decomposition of $H^*(X^k)$.

\item[{\rm (iii)}]
For a generalized partition $\lambda = 
\big (\cdots (-2)^{m_{-2}}(-1)^{m_{-1}} 1^{m_1}2^{m_2} \cdots \big )$, define
\begin{eqnarray*}
\mathfrak a_{\lambda}(\alpha) = \left ( \prod_i (\mathfrak
a_i)^{m_i} \right ) (\alpha)
\end{eqnarray*}
where the product $\prod_i (\mathfrak
a_i)^{m_i} $ is understood to be
$\cdots \mathfrak a_{-2}^{m_{-2}} \mathfrak a_{-1}^{m_{-1}}
 \mathfrak a_{1}^{m_{1}} \mathfrak a_{2}^{m_{2}}\cdots$.
\end{enumerate}
\end{definition}

For $n > 0$ and a homogeneous class $\alpha \in H^*(X)$, let
$|\alpha| = s$ if $\alpha \in H^s(X)$, and let $G_k(\alpha, n)$ be
the homogeneous component in $H^{|\alpha|+2k}(\Xn)$ of
\begin{eqnarray}    \label{DefOfGGammaN}
 G(\alpha, n) = p_{1*}(\ch({\mathcal O}_{{\mathcal Z}_n}) \cdot p_2^*\alpha
\cdot p_2^*{\rm td}(X) ) \in \Hn
\end{eqnarray}
where $\ch({\mathcal O}_{{\mathcal Z}_n})$ denotes the Chern
character of ${\mathcal O}_{{\mathcal Z}_n}$
and ${\rm td}(X) $ denotes the Todd class. We extend the notion $G_k(\alpha,
n)$ linearly to an arbitrary class $\alpha \in H^*(X)$, 
and set $G(\alpha, 0) =0$. 
It was proved in \cite{LQW1} that the cohomology ring of $\Xn$ is
generated by the classes $G_{k}(\alpha, n)$ where $0 \le k < n$
and $\alpha$ runs over a linear basis of $H^*(X)$. 
The {\it Chern character operator} ${\mathfrak G}_k(\alpha) \in
\End({\fock})$ is the operator acting on $H^*(\Xn)$ by the cup product with $G_k(\alpha, n)$. 

The following result is from \cite{LQW2} 
(see also \cite[Theorem~4.7]{Qin1}).

\begin{theorem} \label{char_th}
Let $K_X$ and $e_X$ be the canonical divisor and Euler class of 
the smooth projective surface $X$. Let $k \ge 0$ and $\alpha \in H^*(X)$. 
Then, $\mathfrak G_k(\alpha)$ is equal to
\begin{eqnarray*}
& &- \sum_{\ell(\lambda) = k+2, |\lambda|=0}
   \frac{\mathfrak a_{\lambda}(\alpha)}{\lambda^!} 
   + \sum_{\ell(\lambda) = k, |\lambda|=0}
   \frac{\lambsq - 2}{24}
   \frac{\mathfrak a_{\lambda}(e_X \alpha)}{\lambda^!}  \\
&+&\sum_{\ell(\lambda) = k+1, |\lambda|=0} g_{1, \lambda} 
   \frac{\mathfrak a_{\lambda}(K_X \alpha)}{\la^!} 
   + \sum_{\ell(\lambda) = k, |\lambda|=0} g_{2, \lambda} 
   \frac{\mathfrak a_{\lambda}(K_X^2 \alpha)}{\la^!} 
\end{eqnarray*}
where all the numbers $g_{1, \lambda}$ and $g_{2, \lambda}$ are 
independent of $X$ and $\alpha$. 
\end{theorem}

Closed formulas for the coefficients $g_{1, \lambda}$ and 
$g_{2, \lambda}$ in Theorem~\ref{char_th} are unknown. However, 
recursive formulas to compute $g_{1, \lambda}$ and $g_{2, \lambda}$ 
have been obtained in \cite{Tang}. In particular, 
by \cite[Proposition 5.6]{Tang}, $\fG_2(\alpha)$ is equal to
$$
- \sum_{\ell(\lambda) = 4, |\lambda|=0}
   \frac{\mathfrak a_{\lambda}(\alpha)}{\lambda^!} 
   + \sum_{n > 0} \frac{n^2 - 1}{12}
   \fa_{-n}\fa_n(e_X \alpha)      
$$
\begin{eqnarray}   \label{Tang5.6}
- \sum_{\ell(\lambda) = 3, |\lambda|=0} \frac{|\la_+|-1}{2} 
   \frac{\mathfrak a_{\lambda}(K_X \alpha)}{\la^!}
   - \sum_{n > 0} \frac{(n-1)(2n-1)}{12} \, \fa_{-n}\fa_n(K_X^2 \alpha)
\end{eqnarray}
where $\la_+$ denotes the positive part of a generalized partition 
$\la$. 

\section{\bf The generating series 
$F^{\alpha_1, \ldots, \alpha_N}_{k_1, \ldots, k_N}(q)$} 
\label{sect_CO}

In this section, we will define the generating series 
$F^{\alpha_1, \ldots, \alpha_N}_{k_1, \ldots, k_N}(q)$ in terms of 
the Chern character operators,
and rewrite it by using the Ext vertex operators constructed 
by Carlsson and Okounkov \cite{CO, Car1, Car2}.

Let $L$ be a line bundle over the smooth projective surface $X$. 
Let $\mathbb E_L$ be the virtual vector bundle on $X^{[k]} \times X^{[\ell]}$
whose fiber at $(I, J) \in X^{[k]} \times X^{[\ell]}$ is given by
$$
\mathbb E_L|_{(I,J)} = \chi(\mathcal O, L) - \chi(J, I \otimes L).
$$
Let $\fL_m$ be the trivial line bundle on $X$ with a scaling action of $\C^*$ of 
character $z^m$, and let $\Delta_n$ be the diagonal in $\Xn \times \Xn$. Then, 
\begin{eqnarray}   \label{ResOfEToD}
\mathbb E_{\fL_m}|_{\Delta_n} = T_{\Xn, m},
\end{eqnarray}
the tangent bundle $T_\Xn$ with a scaling action of $\C^*$ of character $z^m$. 
By abusing notations, we also use $L$ to denote its first Chern class. Put
\begin{eqnarray}   
   \Gamma_{\pm}(z) 
&=&\exp \left ( \sum_{n>0} \frac{z^{\mp n}}{n} \fa_{\pm n} \right ), 
        \label{Gammaz}   \\
   \Gamma_{\pm}(L, z) 
&=&\exp \left ( \sum_{n>0} \frac{z^{\mp n}}{n} \fa_{\pm n}(L) \right ).   
        \label{GammaLz}
\end{eqnarray}

\begin{remark}  \label{SignDiff}
There is a sign difference between the Heisenberg commutation relations used in 
\cite{Car1} (see p.3 there) and in this paper (see Theorem~\ref{commutator}).
So for $n > 0$, our Heisenberg operators $\fa_{-n}(L)$ and $\fa_{n}(-L)$
are equal to the Heisenberg operators $\fa_{-n}(L)$ and $\fa_{n}(L)$ in \cite{Car1}.
Accordingly, our operators $\Gamma_-(L, z)$ and $\Gamma_+(-L, z)$ are equal to
the operators $\Gamma_-(L, z)$ and $\Gamma_+(L, z)$ in \cite{Car1}.
\end{remark}

The following commutation relations can be found in \cite{Car1} (see Remark~\ref{SignDiff}):
\begin{eqnarray}   
&[\Gamma_+(L, x), \Gamma_+(L', y)] = [\Gamma_-(L, x), \Gamma_-(L', y)] = 0,&
                                           \label{CarLemma5.1}          \\
&\Gamma_+(L, x)\Gamma_-(L', y) = (1-yx^{-1})^{\langle L, L' \rangle} \,
    \Gamma_-(L', y) \Gamma_+(L, x).&     \label{CarLemma5.2}
\end{eqnarray}

Let $\Wb(L, z): \fock \to \fock$ be the vertex operator constructed in \cite{CO, Car1} 
where $z$ is a formal variable. By \cite{Car1}, $\Wb(L, z)$ is defined via the pairing
\begin{eqnarray}   \label{def-WLz}
  \langle \Wb(L, z) \eta, \xi \rangle 
= z^{\ell-k} \cdot \int_{X^{[k]} \times X^{[\ell]}}
  (\eta \otimes \xi) \, c_{k+\ell}(\mathbb E_L)
\end{eqnarray}
for $\eta \in H^*(X^{[k]})$ and $\xi \in H^*(X^{[\ell]})$.
The main result in \cite{Car1} is (see Remark~\ref{SignDiff}):
\begin{eqnarray}   \label{WLz}
\Wb(L, z) = \Gamma_-(L-K_X, z) \, \Gamma_+(-L, z).
\end{eqnarray}

\begin{definition}
For $\alpha_1, \ldots, \alpha_N \in H^*(X)$ and 
integers $k_1, \ldots, k_N \ge 0$, define 
\begin{eqnarray}    \label{F-generating}
  F^{\alpha_1, \ldots, \alpha_N}_{k_1, \ldots, k_N}(q) 
= \sum_{n \ge 0} q^n \int_\Xn \left ( \prod_{i=1}^N 
  G_{k_i}(\alpha_i, n) \right ) c\big (T_\Xn \big )
\end{eqnarray}
where $c\big (T_\Xn \big )$ is the total Chern class of 
the tangent bundle $T_\Xn$ of $\Xn$. 
\end{definition}

By G\" ottsche's Theorem in \cite{Got}, we have
\begin{eqnarray}    \label{F-generating.1}
  F(q) 
= \sum_{n \ge 0} q^n \, \int_\Xn c(T_\Xn) 
= \sum_{n \ge 0} \chi(\Xn) q^n
= (q; q)_\infty^{-\chi(X)}. 
\end{eqnarray}
Inspired by \cite{Oko}, we define the reduced series
$$
(q; q)_\infty^{\chi(X)} \cdot F^{\alpha_1, \ldots, \alpha_N}_{k_1, \ldots, k_N}(q)
= \frac{F^{\alpha_1, \ldots, \alpha_N}_{k_1, \ldots, k_N}(q)}{F(q)}.
$$
The series $F^{\alpha_1, \ldots, \alpha_N}_{k_1, \ldots, k_N}(q)$
has been studied in \cite{QY}.
The following is \cite[Lemma~3.2]{QY}.

\begin{lemma}  \label{FtoW}
Let $\fn$ be the number-of-points operator, i.e., $\fn|_{H^*(\Xn)} = n \, {\rm Id}$. Then,
\begin{eqnarray}   \label{FtoW.0}
  F^{\alpha_1, \ldots, \alpha_N}_{k_1, \ldots, k_N}(q) 
= \Tr \, q^\fn \, \Wb(\fL_1, z) \, \prod_{i=1}^N \fG_{k_i}(\alpha_i)
\end{eqnarray}
where $\fL_1$ is the trivial line bundle on $X$ with a scaling action of 
$\C^*$ of character $z$.
\end{lemma}

\begin{remark}   \label{RMK_FtoW}
\begin{enumerate}
\item[{\rm (i)}]
In \eqref{FtoW.0}, we have implicitly set $t = 1$ for 
the generator $t$ of the equivariant cohomology $H^*_{\C^*}({\rm pt})$ 
of a point.

\item[{\rm (ii)}]
Let $1_X, K_X, e_X$ be the fundamental class, canonical divisor, 
Euler class of the surface $X$ respectively. 
By \eqref{FtoW.0} and Theorem~\ref{char_th}, 
$F^{\alpha_1, \ldots, \alpha_N}_{k_1, \ldots, k_N}(q)$ is 
an infinite linear combination of the expressions
\begin{eqnarray}   \label{epsilonialphai}
\Tr \, q^\fn \, {\bf W}(\fL_1, z) \, \prod_{i=1}^N 
     \frac{\fa_{\la^{(i)}}(\epsilon_i \alpha_i)}{\la^{(i)!}}
\end{eqnarray}
where $\epsilon_i \in \{1_X, e_X, K_X, K_X^2\}$ 
and $\la^{(1)}, \ldots, \la^{(N)}$ are generalized partitions satisfying 
$|\la^{(i)}| = 0$ and $\ell(\la^{(i)}) = k+2-|\epsilon_i|/2$.
\end{enumerate}
\end{remark}

Let $\lambda = \big (\cdots (-2)^{s_{-2}}(-1)^{s_{-1}} 
1^{s_1}2^{s_2} \cdots \big )$
and $\mu = \big (\cdots (-2)^{t_{-2}}(-1)^{t_{-1}} 
1^{t_1}2^{t_2} \cdots \big )$ be two generalized partitions. 
If $t_i \le s_i$ for every $i$, we write $\mu \le \la$ and define
$$
\la - \mu 
= \big (\cdots (-2)^{s_{-2} - t_{-2}}(-1)^{s_{-1}-t_{-1}} 1^{s_1-t_1}2^{s_2-t_2} \cdots \big ).
$$
The following is from the proof of \cite[Theorem~5.8]{Qin1} 
and deals with \eqref{epsilonialphai}. 

\begin{lemma}   \label{ProofOfQin5.8}
Let $\la^{(1)}, \ldots, \la^{(N)}$ be generalized partitions. Then, 
\begin{eqnarray}     \label{ThmJJkAlpha.3}
& &\Tr \, q^\fn \, {\bf W}(\fL_1, z) \, \prod_{i=1}^N 
     \frac{\fa_{\la^{(i)}}(\alpha_i)}{\la^{(i)!}}   \nonumber   \\
&=&z^{\sum_{i=1}^N |\la^{(i)}|} \cdot \sum_{\substack{\w \la^{(1)} 
     \le \la^{(1)}, \ldots, \w \la^{(N)} \le \la^{(N)}\\
     \sum_{i=1}^N (|\la^{(i)}| - |\w \la^{(i)}|) = 0}}
     \,\, \prod_{\substack{1 \le i \le N\\n \ge 1}} \left ( 
     \frac{(-1)^{p^{(i)}_{n}}}{p^{(i)}_{n}!} 
     \frac{q^{n p^{(i)}_{n}}}{(1-q^n)^{p^{(i)}_{n}}}  
     \frac{1}{\w p^{(i)}_{n}!} \frac{1}{(1-q^n)^{\w p^{(i)}_{n}}} \right )
     \nonumber   \\
& &\cdot \Tr \, q^\fn \prod_{i=1}^N  \frac{\fa_{\la^{(i)} - \w \la^{(i)}}
     \big ((1_X - K_X)^{\sum_{n \ge 1} p^{(i)}_{n}}\alpha_i \big )}
     {\big (\la^{(i)}-\w \la^{(i)} \big )^!}
\end{eqnarray}
where $\w \la^{(i)} = \big ( \cdots (-n)^{\w p^{(i)}_n} \cdots 
(-1)^{\w p^{(i)}_1} 1^{p^{(i)}_1} \cdots n^{p^{(i)}_n} \cdots \big ), 
1 \le i \le N$ are generalized partitions, together with the convention 
that for an empty generalized partition $\mu$,
\begin{eqnarray}     \label{IntBeta}
\frac{\fa_\mu(\beta)}{\mu^!} = \int_X \beta = \langle \beta, 1_X \rangle 
= \langle \beta \rangle.
\end{eqnarray}
\end{lemma}

The leading term in the series 
$F^{\alpha_1, \ldots, \alpha_N}_{k_1, \ldots, 
k_N}(q)$ follows from Remark~\ref{RMK_FtoW}~(ii) and the leading term 
on the the right-hand-side of \eqref{ThmJJkAlpha.3}. 
We refer to the Theorem~5.8 and Theorem~5.12 in \cite{Qin1} for details.
When $N = 1$, the right-hand-side of \eqref{ThmJJkAlpha.3} 
has been worked out explicitly in \cite[Remark~5.9]{Qin1}. 

\begin{lemma}   \label{Qin5.9}
Let $\la = (\cdots (-n)^{\w m_n} \cdots 
(-1)^{\w m_1} 1^{m_1}\cdots n^{m_n} \cdots)$ be a generalized partition. 
For $n_1 \ge 1$ with $m_{n_1} \cdot \w m_{n_1} \ge 1$, define
$m_{n_1}(n_1) = m_{n_1}-1$, $\w m_{n_1}(n_1) = \w m_{n_1}-1$,
and $m_n(n_1) = m_n$, $\w m_n(n_1) = \w m_n$ if $n \ne n_1$. 
Then, 
$
\displaystyle{\Tr \, q^\fn \, {\bf W}(\fL_1, z) \, 
\frac{\fa_\la(\alpha)}{\la^!}}
$ 
equals
$$
{z^{|\la|}} {(q; q)_\infty^{-\chi(X)}} \cdot
   \langle (1_X - K_X)^{\sum_{n \ge 1} m_{n}}, \alpha \rangle   \cdot \prod_{n \ge 1} \left ( \frac{(-1)^{m_{n}}}{m_{n}!} \frac{q^{n m_n}}{(1-q^n)^{m_n}} 
   \frac{1}{\w m_{n}!} \frac{1}{(1-q^n)^{\w m_n}} \right )
$$
\begin{eqnarray*}    
&+ &{z^{|\la|}} {(q; q)_\infty^{-\chi(X)}} \cdot \langle e_X, \alpha \rangle 
   \cdot \sum_{n_1 \ge 1 \, \text{\rm with } m_{n_1} \cdot \w m_{n_1} \ge 1} 
   \frac{(-n_1) q^{n_1}}{1-q^{n_1}} \cdot                  \nonumber   \\
& &\cdot \prod_{n \ge 1} \left ( \frac{(-1)^{m_{n}(n_1)}}{m_{n}(n_1)!} 
   \frac{q^{n m_n(n_1)}}{(1-q^n)^{m_n(n_1)}} 
   \frac{1}{\w m_{n}(n_1)!} \frac{1}{(1-q^n)^{\w m_n(n_1)}} \right ).              
\end{eqnarray*}
\end{lemma}

\section{\bf Computation of $(q;q)_\infty^{\chi(X)} \cdot 
F_2^{\alpha}(q)$} 
\label{sect_F2Alpha}

In this section, we will compute the reduced series 
$(q;q)_\infty^{\chi(X)} \cdot F_2^{\alpha}(q)$ 
where $\alpha \in H^*(X)$ with $X$ denoting a smooth projective surface. 
The results will be used in next section 
to partially verified Okounkov's Conjecture. 
\subsection{\bf A preliminary formula for $(q;q)_\infty^{\chi(X)} \cdot 
F_2^{\alpha}(q)$} 
\label{subsect_preliminary}
$\,$
\par

In this subsection, we will obtain a preliminary formula for 
$(q;q)_\infty^{\chi(X)} \cdot F_2^{\alpha}(q)$.
By \eqref{FtoW.0} and \eqref{Tang5.6}, $F_2^{\alpha}(q)$ is equal to
the sum of the next four lines:
\begin{eqnarray}   \label{F2Alpha.1}
\sum_{n > 0} \frac{n^2 - 1}{12}
   \Tr \, q^\fn \, {\bf W}(\fL_1, z) \fa_{-n}\fa_n(e_X\alpha)      
\end{eqnarray}
\begin{eqnarray}   \label{F2Alpha.2}
- \sum_{n > 0} \frac{(n-1)(2n-1)}{12} \, 
\Tr \, q^\fn \, {\bf W}(\fL_1, z) \fa_{-n}\fa_n(K_X^2\alpha)
\end{eqnarray}
\begin{eqnarray}   \label{F2Alpha.3}
- \sum_{\ell(\lambda) = 3, |\lambda|=0} \frac{|\la_+|-1}{2}
\Tr \, q^\fn \, {\bf W}(\fL_1, z) \frac{\fa_{\la}(K_X\alpha)}{\la^!}
\end{eqnarray}
\begin{eqnarray}   \label{F2Alpha.4}
- \sum_{\ell(\lambda) = 4, |\lambda|=0}
   \Tr \, q^\fn \, {\bf W}(\fL_1, z)
   \frac{\mathfrak a_{\lambda}(\alpha)}{\lambda^!}.
\end{eqnarray}

The first two lines \eqref{F2Alpha.1} and \eqref{F2Alpha.2} 
can be calculated easily. 

\begin{lemma}   \label{first2terms}
Let $\alpha \in H^*(X)$. Then, 
\begin{enumerate}
\item[{\rm (i)}]
\eqref{F2Alpha.1} is equal to
$$
-(q;q)_\infty^{-\chi(X)} \cdot \langle e_X, \alpha \rangle \cdot 
\sum_{n > 0} \frac{n^2 - 1}{12} \frac{q^n}{(1-q^n)^2}.
$$

\item[{\rm (ii)}]
\eqref{F2Alpha.2} is equal to
$$
(q;q)_\infty^{-\chi(X)} \cdot \langle K_X^2, \alpha \rangle \cdot 
\sum_{n > 0} \frac{(n-1)(2n-1)}{12} \frac{q^n}{(1-q^n)^2}.
$$
\end{enumerate}
\end{lemma}
\begin{proof}
By Lemma~\ref{Qin5.9}, 
$\Tr \, q^\fn \, {\bf W}(\fL_1, z) \fa_{-n}\fa_n(\beta)$ is equal to
$$
(q;q)_\infty^{-\chi(X)} \cdot \langle 1_X - K_X, \beta \rangle 
  \cdot \frac{-q^n}{(1-q^n)^2}
  + (q;q)_\infty^{-\chi(X)} \cdot \langle e_X, \beta \rangle 
  \cdot \frac{-nq^n}{1-q^n}
$$
when $n$ is a positive integer.
Now both (i) and (ii) follow immediately.
\end{proof}

\begin{lemma}   \label{3rdterms}
Let $\alpha \in H^*(X)$. Then, \eqref{F2Alpha.3} is equal to
$$
(q;q)_\infty^{-\chi(X)} \cdot \frac{\langle K_X^2, \alpha \rangle}2
   \sum_{i,j \ge 1} \frac{i+j-1}{2} \cdot
   \frac{q^{i+j}}{(1-q^i)(1-q^j)(1-q^{i+j})}.
$$
\end{lemma}
\begin{proof}
For $\ell(\lambda) = 3$ and $|\lambda|=0$, we see from 
Lemma~\ref{Qin5.9} that 
\begin{eqnarray*}
& &\Tr \, q^\fn \, {\bf W}(\fL_1, z) \mathfrak a_{\lambda}(K_X\alpha)\\
&=&\begin{cases}
  (q;q)_\infty^{-\chi(X)} \cdot 
    \frac{\langle K_X^2-K_X, \alpha \rangle \cdot 
    q^{i+j}}{(1-q^i)(1-q^j)(1-q^{i+j})}          
    &\text{if $\la = (-j, -i, i+j)$,} \\
  (q;q)_\infty^{-\chi(X)} \cdot  
    \frac{\langle K_X-2K_X^2, \alpha \rangle \cdot 
    q^{i+j}}{(1-q^i)(1-q^j)(1-q^{i+j})} 
    &\text{if $\la = (-i-j, i, j)$}
  \end{cases}
\end{eqnarray*}
where $i$ and $j$ denote positive integers with $i \le j$.
So \eqref{F2Alpha.3} is equal to
\begin{eqnarray*}   
& &-(q;q)_\infty^{-\chi(X)} \cdot
   \sum_{1 \le i \le j} \frac{i+j-1}{2} \cdot\frac1{1+\delta_{i,j}} 
   \cdot \frac{\langle K_X^2-K_X, \alpha \rangle \cdot 
    q^{i+j}}{(1-q^i)(1-q^j)(1-q^{i+j})}  \\
& &-(q;q)_\infty^{-\chi(X)} \cdot
   \sum_{1 \le i \le j} \frac{i+j-1}{2} \cdot \frac1{1+\delta_{i,j}} 
   \cdot \frac{\langle K_X-2K_X^2, \alpha \rangle \cdot 
    q^{i+j}}{(1-q^i)(1-q^j)(1-q^{i+j})}  \\
&=&-(q;q)_\infty^{-\chi(X)} \cdot \frac12
   \sum_{i,j \ge 1} \frac{i+j-1}{2} \cdot
   \frac{\langle K_X^2-K_X, \alpha \rangle \cdot 
    q^{i+j}}{(1-q^i)(1-q^j)(1-q^{i+j})}  \\
& &-(q;q)_\infty^{-\chi(X)} \cdot \frac12
   \sum_{i,j \ge 1} \frac{i+j-1}{2} \cdot 
   \frac{\langle K_X-2K_X^2, \alpha \rangle \cdot 
    q^{i+j}}{(1-q^i)(1-q^j)(1-q^{i+j})}  \\
&=&(q;q)_\infty^{-\chi(X)} \cdot \frac{\langle K_X^2, \alpha \rangle}2
   \sum_{i,j \ge 1} \frac{i+j-1}{2} \cdot
   \frac{q^{i+j}}{(1-q^i)(1-q^j)(1-q^{i+j})} 
\end{eqnarray*}
where $\delta_{i,j} = 1$ if $i=j$ and $\delta_{i,j} = 0$ if $i \ne j$.
\end{proof}

\begin{proposition}   \label{prop4thterms}
Let $\alpha \in H^*(X)$. Then, 
$(q;q)_\infty^{\chi(X)} \cdot F_2^{\alpha}(q)$ is equal to
$$
-\langle e_X, \alpha \rangle \cdot 
\sum_{n > 0} \frac{n^2 - 1}{12} \frac{q^n}{(1-q^n)^2}
+ \langle K_X^2, \alpha \rangle \cdot 
\sum_{n > 0} \frac{(n-1)(2n-1)}{12} \frac{q^n}{(1-q^n)^2}
$$
$$
+ \frac{\langle K_X^2, \alpha \rangle}2 \cdot 
   \sum_{i,j > 0} \frac{i+j-1}{2} \cdot
   \frac{q^{i+j}}{(1-q^i)(1-q^j)(1-q^{i+j})}
$$
$$
+ \langle 1_X-K_X, \alpha \rangle \cdot 
    \sum_{\substack{\la = (-i,-j,-k, i+j+k)\\i \ge j \ge k > 0}} 
    \frac{1}{\la^!} \frac{q^{i+j+k}}{(1-q^{i+j+k})(1-q^i)(1-q^j)(1-q^k)}
$$
$$
+ \langle (1_X-K_X)^3, \alpha \rangle \cdot 
    \sum_{\substack{\la = (-i-j-k, k, j, i)\\i \ge j \ge k > 0}} 
    \frac{1}{\la^!} \frac{q^{i+j+k}}{(1-q^i)(1-q^j)(1-q^k)(1-q^{i+j+k})}
$$

$$
- \langle (1_X-K_X)^2, \alpha \rangle \cdot 
   \sum_{\substack{\la = (-i, -j, \ell, k)\\i \ge j > 0, 
      k \ge \ell > 0, i+j=k+\ell}} 
   \frac{1}{\la^!} \frac{q^{i+j}}{(1-q^i)(1-q^j)(1-q^k)(1-q^{\ell})}
$$
$$
- \langle e_X, \alpha \rangle \cdot \sum_{i > j > 0} 
    \left (\frac{iq^i}{1-q^i} \cdot \frac{q^j}{(1-q^j)^2} 
    + \frac{jq^j}{1-q^j} \cdot \frac{q^i}{(1-q^i)^2}\right )
- \langle e_X, \alpha \rangle \cdot \sum_{i > 0} 
    \frac{iq^{2i}}{(1-q^i)^3}.
$$
\end{proposition}
\begin{proof}
First of all, recall that $F_2^{\alpha}(q)$ is equal to the sum of 
\eqref{F2Alpha.1}, \eqref{F2Alpha.2}, \eqref{F2Alpha.3} and 
\eqref{F2Alpha.4}. So by Lemma~\ref{first2terms} and 
Lemma~\ref{3rdterms}, 
$(q;q)_\infty^{\chi(X)} \cdot F_2^{\alpha}(q)$ is equal to
$$
-\langle e_X, \alpha \rangle \cdot 
\sum_{n > 0} \frac{n^2 - 1}{12} \frac{q^n}{(1-q^n)^2}
+ \langle K_X^2, \alpha \rangle \cdot 
\sum_{n > 0} \frac{(n-1)(2n-1)}{12} \frac{q^n}{(1-q^n)^2}
$$
$$
+ \frac{\langle K_X^2, \alpha \rangle}2 \cdot 
   \sum_{i,j \ge 1} \frac{i+j-1}{2} \cdot
   \frac{q^{i+j}}{(1-q^i)(1-q^j)(1-q^{i+j})}
$$
\begin{eqnarray}   \label{prop4thterms.100}
- (q;q)_\infty^{\chi(X)} \cdot \sum_{\ell(\lambda) = 4, |\lambda|=0}
   \Tr \, q^\fn \, {\bf W}(\fL_1, z)
   \frac{\mathfrak a_{\lambda}(\alpha)}{\lambda^!}.
\end{eqnarray}

It remains to compute line \eqref{prop4thterms.100}.
Let $\ell(\lambda) = 4$ and $|\lambda|=0$.
If $\ell(\lambda_+) = 1$, then put $\la = (-i, -j, -k, i+ j+ k)$ with 
$i \ge j \ge k > 0$. In this case, Lemma~\ref{Qin5.9} yields
\begin{eqnarray}    \label{prop4thterms.1}
& &(q;q)_\infty^{\chi(X)} \cdot \Tr \, q^\fn \, {\bf W}(\fL_1, z) 
    \mathfrak a_{\lambda}(\alpha)   \nonumber   \\
&=&\langle 1_X-K_X, \alpha \rangle \cdot 
    \frac{-q^{i+j+k}}{1-q^{i+j+k}} \cdot
    \frac{1}{(1-q^i)(1-q^j)(1-q^k)}. 
\end{eqnarray} 
If $\ell(\lambda_+) = 3$, then put $\la = (-i-j-k, k, j, i)$ with 
$i \ge j \ge k > 0$. In this case, 
\begin{eqnarray}    \label{prop4thterms.2}
& &(q;q)_\infty^{\chi(X)} \cdot \Tr \, q^\fn \, {\bf W}(\fL_1, z) 
    \mathfrak a_{\lambda}(\alpha)   \nonumber   \\
&=&\langle (1_X-K_X)^3, \alpha \rangle \cdot 
    \frac{-q^{i+j+k}}{(1-q^i)(1-q^j)(1-q^k)} \cdot \frac{1}{1-q^{i+j+k}}. 
\end{eqnarray} 
If $\ell(\lambda_+) = 2$, then put $\la = (-i, -j, \ell, k)$ with 
$i \ge j > 0, k \ge \ell > 0$ and $i + j = k + \ell$. In this case,
$(q;q)_\infty^{\chi(X)} \cdot \Tr \, q^\fn \, {\bf W}(\fL_1, z) 
\mathfrak a_{\lambda}(\alpha)$ is equal to
$$
\langle (1_X-K_X)^2, \alpha \rangle \cdot 
\frac{q^{i+j}}{(1-q^i)(1-q^j)(1-q^k)(1-q^{\ell})}
$$
\begin{eqnarray}    \label{prop4thterms.3}
+ \begin{cases}
  0&\text{if $\{i, j\} \cap \{k, \ell\} = \emptyset$,} \\
  \langle e_X, \alpha \rangle \cdot \left (\frac{-iq^i}{1-q^i} \cdot 
    \frac{-q^j}{(1-q^j)^2} 
    + \frac{-jq^j}{1-q^j} \cdot \frac{-q^i}{(1-q^i)^2}\right ) 
    &\text{if $i = k > j = \ell$,}  \\
  \la^! \cdot \langle e_X, \alpha \rangle \cdot \frac{-iq^i}{1-q^i} \cdot 
    \frac{-q^i}{(1-q^i)^2}&\text{if $i = j = k = \ell$.}
  \end{cases}
\end{eqnarray}
Combining line \eqref{prop4thterms.100}, \eqref{prop4thterms.1}, \eqref{prop4thterms.2} and \eqref{prop4thterms.3} yields our proposition.
\end{proof}
\subsection{\bf Some identities for multiple $q$-zeta values} 
\label{subsect_identities}
$\,$
\par

In this subsection, we will prove some identifies involving 
multiple $q$-zeta values. In next subsection, we will use 
these identifies to simplify the expression of 
$(q;q)_\infty^{\chi(X)} \cdot F_2^{\alpha}(q)$ presented in 
Proposition~\ref{prop4thterms}.

\begin{lemma}  \label{Z2k2k}
Let $k$ be a positive integer. Then,
$$
Z(2k, 2k) = -\frac12 Z(4k) + \frac12 Z(2k)^2.
$$
\end{lemma}
\begin{proof}
By Definition~\ref{def_qMZV}~(iv), we have
\begin{eqnarray*} 
   Z(2k, 2k)  
&=&\sum_{n>m>0} \frac{q^{nk}}{(1-q^n)^{2k}} \frac{q^{mk}}{(1-q^m)^{2k}} \\
&=&\frac12 \sum_{n \ne m, n, m>0} \frac{q^{nk}}{(1-q^n)^{2k}} 
    \frac{q^{mk}}{(1-q^m)^{2k}}   \\
&=&-\frac12 \sum_{n >0} \frac{q^{n(2k)}}{(1-q^n)^{4k}} 
   + \frac12 \sum_{n, m>0} \frac{q^{nk}}{(1-q^n)^{2k}} 
     \frac{q^{mk}}{(1-q^m)^{2k}}.
\end{eqnarray*}
It follows that $\displaystyle{Z(2k, 2k) 
= -\frac12 Z(4k) + \frac12 Z(2k)^2}$.
\end{proof}

\begin{lemma} \label{lemma_n2qn}
\begin{eqnarray}  
\sum_{n > 0} \frac{n q^n}{1-q^n}
  &=&Z(2),    \label{lemma_n2qn.01}    \\
\sum_{n > 0} \frac{n^2 q^n}{(1-q^n)^2}
  &=&\frac72 Z(4) - \frac12 Z(2)^2 + Z(2),    \label{lemma_n2qn.02}   \\
\sum_{n > 0} \frac{n q^{2n} + nq^n}{(1-q^n)^3}
  &=&\frac72 Z(4) - \frac12 Z(2)^2 + Z(2).    \label{lemma_n2qn.03}
\end{eqnarray}
\end{lemma}
\begin{proof}
First of all, we have 
$$
\sum_{n > 0} \frac{n q^n}{1-q^n}
= \sum_{n, d > 0} nq^{nd} = \sum_{n, d > 0} dq^{nd}.
$$
So \eqref{lemma_n2qn.01} follows from \eqref{def_qMZV.04}.
Next, we notice that 
$$
  \sum_{n > 0} \frac{n^2 q^n}{(1-q^n)^2}   
= q\frac{\rm d}{{\rm d}q} \sum_{n > 0} \frac{nq^n}{1-q^n}  
= q\frac{\rm d}{{\rm d}q} Z(2).
$$
By \cite[Example~2.11]{BK3} and Lemma~\ref{Z2k2k}, we obtain
\begin{eqnarray}   \label{2.11BK3}
  q\frac{\rm d}{{\rm d}q} Z(2) 
= 3Z(4) - Z(2, 2) + Z(2)
= \frac72 Z(4) - \frac12 Z(2)^2 + Z(2).
\end{eqnarray}
Hence \eqref{lemma_n2qn.02} follows immediately. Finally, we have
$$
  \sum_{n > 0} \frac{n q^{2n} + nq^n}{(1-q^n)^3}
= q\frac{\rm d}{{\rm d}q} \sum_{n > 0} \frac{q^n}{(1-q^n)^2}
= q\frac{\rm d}{{\rm d}q} Z(2).
$$
Therefore, \eqref{lemma_n2qn.03} follows from \eqref{2.11BK3} again.
\end{proof}

Our lemma below expresses the series $\displaystyle{\sum_{i, j > 0} 
\frac{(i+j)q^{i+j}}{(1-q^i)(1-q^j)(1-q^{i+j})}}$ in terms of 
the basic series $[1]$ which can be found in \eqref{example[s]}.

\begin{lemma} \label{ijkCO}
Let $[1]$ be the series from \eqref{example[s]}. Then, 
\begin{eqnarray}  
\sum_{i, j > 0} \frac{q^{i+j}}{(1-q^i)(1-q^j)(1-q^{i+j})}
  &=&Z(3) - q\frac{\rm d}{{\rm d}q}[1],    \label{ijkCO.01}    \\
\sum_{i, j > 0} \frac{(i+j)q^{i+j}}{(1-q^i)(1-q^j)(1-q^{i+j})}
  &=&2 \sum_{n > m > 0} \frac{nq^n}{(1-q^n)^2(1-q^m)}  
      + q\frac{\rm d}{{\rm d}q}[1]    \nonumber  \\
  & &-\frac72 Z(4) + \frac12 Z(2)^2 - Z(2). \label{ijkCO.02}
\end{eqnarray}
\end{lemma}
\begin{proof}
Formula \eqref{ijkCO.01} is simply \cite[Lemma~4]{CO} by noting that 
the series $E_3$ and $E_1$ there are equal to our $Z(3)$ and $[1]$ 
respectively. To prove \eqref{ijkCO.02}, notice that 
\begin{eqnarray}   \label{qiqj}
  \frac{1}{(1-q^i)(1-q^j)}
= \left (\frac{1}{1-q^i} + \frac{q^j}{1-q^j} \right ) \frac{1}{1-q^{i+j}} 
\end{eqnarray}
for positive integers $i$ and $j$. Therefore, we obtain
\begin{eqnarray*}   
& &\sum_{i, j > 0} \frac{(i+j)q^{i+j}}{(1-q^i)(1-q^j)(1-q^{i+j})}   \\
&=&\sum_{i, j > 0} \frac{(i+j)q^{i+j}}{(1-q^i)(1-q^{i+j})^2} 
   + \sum_{i, j > 0} \frac{(i+j)q^{i+2j}}{(1-q^j)(1-q^{i+j})^2} \\
&=&\sum_{n > m > 0} \frac{nq^n}{(1-q^n)^2(1-q^m)} 
   + \sum_{n > m > 0} \frac{nq^{n+m}}{(1-q^n)^2(1-q^m)}.
\end{eqnarray*}
Since $q^m/(1-q^m) = 1/(1-q^m) - 1$, we see that 
\begin{eqnarray*}   
& &\sum_{i, j > 0} \frac{(i+j)q^{i+j}}{(1-q^i)(1-q^j)(1-q^{i+j})}   \\
&=&2 \sum_{n > m > 0} \frac{nq^n}{(1-q^n)^2(1-q^m)} 
   - \sum_{n > 0} \frac{n(n-1)q^{n}}{(1-q^n)^2}  \\
&=&2 \sum_{n > m > 0} \frac{nq^n}{(1-q^n)^2(1-q^m)} 
   + q\frac{\rm d}{{\rm d}q}[1]
   - \sum_{n > 0} \frac{n^2 q^{n}}{(1-q^n)^2}.
\end{eqnarray*}
Now \eqref{ijkCO.02} follows from \eqref{lemma_n2qn.02}.
\end{proof}

In the next lemma, we reduce 
$\displaystyle{\sum_{i,j,k, \ell >0, i+j=k+\ell}   
   \frac{q^{i+j}}{(1-q^i)(1-q^j)(1-q^k)(1-q^{\ell})}}$ to 
$$
\sum_{i,j,k > 0} \frac{q^{i+j+k}}{(1-q^i)(1-q^j)(1-q^k)(1-q^{i+j+k})}.
$$

\begin{lemma} \label{i+j=k+l}
$\displaystyle{\sum_{i,j,k, \ell >0, i+j=k+\ell}   
   \frac{q^{i+j}}{(1-q^i)(1-q^j)(1-q^k)(1-q^{\ell})}}$ is equal to
\begin{eqnarray}  \label{i+j=k+l.0}
\frac43 Z(2)^2 - \frac13 Z(4) + \frac43 \sum_{i,j,k > 0} 
\frac{q^{i+j+k}}{(1-q^i)(1-q^j)(1-q^k)(1-q^{i+j+k})}.
\end{eqnarray}
\end{lemma}
\begin{proof}
By \cite[Corollary~4.9~(i)]{SQ}, we have a quasi-modular form 
$\Theta_2(q)$ of weight $4$. Moreover,  
by \cite[Lemma~2.3~(i) and (ii)]{SQ}, 
$$
  \Theta_2(q)
= \Coe_{z^0} \frac{1}{4!} \left (\sum_{m > 0} \frac{(q z)^{m}}{1-q^{m}}
  - \sum_{m > 0} \frac{z^{-m}}{1-q^{m}} \right )^4
$$
where $\Coe_{z^0}(\cdot)$ denotes the coefficient of $z^0$.
Expanding the right-hand-side yields
\begin{eqnarray}     \label{i+j=k+l.1}
   \Theta_2(q)
&=&-\frac13 \sum_{i,j,k > 0} \frac{q^{i+j+k}}
      {(1-q^i)(1-q^j)(1-q^k)(1-q^{i+j+k})}  \nonumber  \\
& &+ \,\, \frac14 \sum_{i,j,k, \ell >0, i+j=k+\ell}   
   \frac{q^{i+j}}{(1-q^i)(1-q^j)(1-q^k)(1-q^{\ell})}.
\end{eqnarray}
Expanding the right-hand-side of \eqref{i+j=k+l.1} further, we obtain 
\begin{eqnarray}     \label{i+j=k+l.2}
\Theta_2(q) = \frac14 q^2+ \frac53 q^3+ \frac{19}4 q^4+11q^5 +O(q^6)
\end{eqnarray}
where $O(q^k)$ denotes the terms with degree of $q$ being at least $k$.

On the other hand, by \eqref{20170812522pm}, the quasi-modular form $\Theta_2(q)$ of 
weight $4$ can be written as 
a linear combination of $1, Z(2), Z(2)^2$ and $Z(4)$:
\begin{eqnarray}     \label{i+j=k+l.3}
  \Theta_2(q)
= r_1 + r_2 \cdot Z(2) + r_3 \cdot Z(2)^2 + r_4 \cdot Z(4)
\end{eqnarray}
where $r_1, r_2, r_3, r_4 \in \Q$. Expanding the right-hand-side of 
\eqref{i+j=k+l.3} and using \eqref{i+j=k+l.2} enable us to 
determine the coefficients $r_1, r_2, r_3, r_4$. We have
\begin{eqnarray}     \label{i+j=k+l.4}
\Theta_2(q) = \frac13 Z(2)^2 - \frac1{12} Z(4).
\end{eqnarray}
Now our lemma follows from \eqref{i+j=k+l.1} and \eqref{i+j=k+l.4}.
\end{proof}

Our last lemma in this subsection calculates the series 
$$
\sum_{i,j,k > 0} \frac{q^{i+j+k}}{(1-q^i)(1-q^j)(1-q^k)(1-q^{i+j+k})}.
$$
The proof is a little long, but the main idea is to 
apply \eqref{qiqj} repeatedly.

\begin{lemma} \label{lemma_ijk100}
$\displaystyle{\sum_{i,j,k > 0} \frac{q^{i+j+k}}{(1-q^i)(1-q^j)
(1-q^k)(1-q^{i+j+k})}}$ is equal to 
$$
-3 \sum_{n> m > 0} \frac{nq^{n}}{(1-q^{n})^2(1-q^{m})}
   - \frac{3}{2} q\frac{\rm d}{{\rm d}q}[1]
   + \frac{31}{4} Z(4) - \frac{1}{4} Z(2)^2 + \frac{3}{2} Z(2)  
$$
where $[1]$ is the series from \eqref{example[s]}.
\end{lemma}
\begin{proof}
For simplicity, we use {\rm LHS} to denote 
$$
\sum_{i,j,k > 0} \frac{q^{i+j+k}}{(1-q^i)(1-q^j)(1-q^k)(1-q^{i+j+k})}.
$$
By \eqref{qiqj}, {\rm LHS} is equal to
$$
\sum_{i,j,k > 0} \frac{q^{i+j+k}}{(1-q^i)(1-q^k)(1-q^{i+j})
     (1-q^{i+j+k})}   
$$
$$
+ \sum_{i,j,k > 0} \frac{q^{i+2j+k}}{(1-q^j)(1-q^{i+j})(1-q^k)
     (1-q^{i+j+k})}.
$$
Applying \eqref{qiqj} repeatedly, we see that {\rm LHS} is equal to
\begin{eqnarray*}
& &\sum_{i,j,k > 0} \frac{q^{i+j+k}}{(1-q^i)(1-q^k)(1-q^{i+j+k})^2}   \\
& &+ \sum_{i,j,k > 0} \frac{q^{2i+2j+k}}{(1-q^i)(1-q^{i+j})
     (1-q^{i+j+k})^2}   \\
& &+ \sum_{i,j,k > 0} \frac{q^{i+2j+k}}{(1-q^j)(1-q^{i+j})
     (1-q^{i+j+k})^2}  \\
& &+ \sum_{i,j,k > 0} \frac{q^{i+2j+2k}}{(1-q^j)(1-q^k)(1-q^{i+j+k})^2} 
\end{eqnarray*}
which in turn is equal to 
\begin{eqnarray*}
& &\sum_{i,j,k > 0} \frac{q^{i+j+k}}{(1-q^i)(1-q^{i+k})
     (1-q^{i+j+k})^2}   \\
& &+ \sum_{i,j,k > 0} \frac{q^{i+j+2k}}{(1-q^k)(1-q^{i+k})
     (1-q^{i+j+k})^2}   \\
& &+ \sum_{i,j,k > 0} \frac{q^{2i+2j+k}}{(1-q^i)(1-q^{i+j})
     (1-q^{i+j+k})^2}   \\
& &+ \sum_{i,j,k > 0} \frac{q^{i+2j+k}}{(1-q^j)(1-q^{i+j})
     (1-q^{i+j+k})^2}  \\
& &+ \sum_{i,j,k > 0} \frac{q^{i+2j+2k}}{(1-q^j)(1-q^{j+k})
     (1-q^{i+j+k})^2}\\
& &+ \sum_{i,j,k > 0} \frac{q^{i+2j+3k}}{(1-q^k)(1-q^{j+k})
     (1-q^{i+j+k})^2}.
\end{eqnarray*}
After changing of variables, we conclude that 
\begin{eqnarray*}
   {\rm LHS}
&=&\sum_{n_1>n_2>n_3 > 0} \frac{q^{n_1}}{(1-q^{n_1})^2(1-q^{n_2})
     (1-q^{n_3})}   \\
& &+ 2 \sum_{n_1>n_2>n_3 > 0} \frac{q^{n_1+n_2}}{(1-q^{n_1})^2(1-q^{n_2})
     (1-q^{n_3})}   \\
& &+ 2 \sum_{n_1>n_2>n_3 > 0} \frac{q^{n_1+n_3}}{(1-q^{n_1})^2(1-q^{n_2})
     (1-q^{n_3})}  \\
& &+ \sum_{n_1>n_2>n_3 > 0} \frac{q^{n_1+n_2+n_3}}{(1-q^{n_1})^2(1-q^{n_2})
     (1-q^{n_3})}. 
\end{eqnarray*}
Since $q^n/(1-q^n) = 1/(1-q^n) - 1$, {\rm LHS} is equal to
\begin{eqnarray*} 
& &\sum_{n_1>n_2>n_3 > 0} \frac{q^{n_1}}{(1-q^{n_1})^2(1-q^{n_2})
     (1-q^{n_3})}   \\
& &+ 2 \sum_{n_1>n_2>n_3 > 0} \frac{q^{n_1}}{(1-q^{n_1})^2
     (1-q^{n_3})} \left (\frac1{1-q^{n_2}} - 1 \right )   \\
& &+ 2 \sum_{n_1>n_2>n_3 > 0} \frac{q^{n_1}}{(1-q^{n_1})^2(1-q^{n_2})} 
     \left (\frac1{1-q^{n_3}} - 1 \right )  \\
& &+ \sum_{n_1>n_2>n_3 > 0} \frac{q^{n_1}}{(1-q^{n_1})^2}
     \left (\frac1{1-q^{n_2}} - 1 \right ) 
     \left (\frac1{1-q^{n_3}} - 1 \right ).
\end{eqnarray*}
Simplifying the above, we see that {\rm LHS} is equal to
\begin{eqnarray*} 
& &6 \sum_{n_1>n_2>n_3 > 0} \frac{q^{n_1}}{(1-q^{n_1})^2(1-q^{n_2})
     (1-q^{n_3})}   \\
& &- 3 \sum_{n_1> n_3 > 0} \frac{q^{n_1}}{(1-q^{n_1})^2
     (1-q^{n_3})} (n_1 - n_3 - 1)   \\
& &- 3 \sum_{n_1> n_2 > 0} \frac{q^{n_1}}{(1-q^{n_1})^2
     (1-q^{n_2})} (n_2 - 1)  \\
& &+ \sum_{n_1 > 0} \frac{q^{n_1}}{(1-q^{n_1})^2}
     \frac{(n_1-1)(n_1-2)}2  \\
&=&6 \sum_{n_1>n_2>n_3 > 0} \frac{q^{n_1}}{(1-q^{n_1})^2(1-q^{n_2})
     (1-q^{n_3})}   \\
& &- 3 \sum_{n_1> n_2 > 0} \frac{(n_1 - 2)q^{n_1}}{(1-q^{n_1})^2
     (1-q^{n_2})} 
     + \frac12 \sum_{n > 0} \frac{(n-1)(n-2)q^{n}}{(1-q^{n})^2}.
\end{eqnarray*}

Next, by \cite[Corollary~4]{Bra1}, we have
\begin{eqnarray}
\sum_{n_1 > n_2 > 0} \frac{q^{n_1}}{(1-q^{n_1})^2} \frac{1}{1-q^{n_2}}
  &=&\sum_{n > 0} \frac{q^{2n}}{(1-q^n)^3}   \label{Bra1Cor4.1}  \\
\sum_{n_1>n_2>n_3 > 0} \frac{q^{n_1}}{(1-q^{n_1})^2(1-q^{n_2})(1-q^{n_3})}
  &=&\sum_{n > 0} \frac{q^{3n}}{(1-q^n)^4}.   \label{Bra1Cor4.2}
\end{eqnarray}
By Lemma~\ref{lemma_n2qn}, we obtain
\begin{eqnarray*}   
& &\frac{1}{2}\sum_{n > 0} \frac{(n-1)(n-2)q^n}{(1-q^n)^2} \\
&=&\frac{1}{2} (3Z(4) - Z(2, 2) + Z(2)) 
   - \frac32 \sum_{n > 0} \frac{nq^n}{(1-q^n)^2} + Z(2) \\
&=&\frac{3}{2} Z(4) - \frac{1}{2} Z(2, 2) + \frac{3}{2} Z(2) 
   - \frac{3}{2} q\frac{\rm d}{{\rm d}q}[1].
\end{eqnarray*}
Therefore, {\rm LHS} is equal to
\begin{eqnarray*} 
& &6 \sum_{n > 0} \frac{q^{3n}}{(1-q^n)^4}   
   + 6 \sum_{n > 0} \frac{q^{2n}}{(1-q^n)^3}
   - 3 \sum_{n_1> n_2 > 0} \frac{n_1q^{n_1}}{(1-q^{n_1})^2(1-q^{n_2})} \\
& &+ \frac{3}{2} Z(4) - \frac{1}{2} Z(2, 2) + \frac{3}{2} Z(2) 
   - \frac{3}{2} q\frac{\rm d}{{\rm d}q}[1]  \\
&=&6 \sum_{n > 0} \frac{q^{2n}}{(1-q^n)^4}   
   - 3 \sum_{n> m > 0} \frac{nq^{n}}{(1-q^{n})^2(1-q^{m})} 
   - \frac{3}{2} q\frac{\rm d}{{\rm d}q}[1] \\
& &+ \frac{3}{2} Z(4) - \frac{1}{2} Z(2, 2) + \frac{3}{2} Z(2)   \\
&=&-3 \sum_{n> m > 0} \frac{nq^{n}}{(1-q^{n})^2(1-q^{m})}
   - \frac{3}{2} q\frac{\rm d}{{\rm d}q}[1]
   + \frac{15}{2} Z(4) - \frac{1}{2} Z(2, 2) + \frac{3}{2} Z(2).
\end{eqnarray*}
Combining this with Lemma~\ref{Z2k2k} completes the proof of our lemma.
\end{proof}
\subsection{\bf The formula for $(q;q)_\infty^{\chi(X)} \cdot 
F_2^{\alpha}(q)$} 
\label{subsect_formula}
$\,$
\par

In the present subsection, we will simplify the expression in 
Proposition~\ref{prop4thterms}. We will achieve this by simplifying 
the coefficients of $\langle e_X, \alpha \rangle, 
\langle 1_X, \alpha \rangle, \langle K_X, \alpha \rangle, 
\langle K_X^2, \alpha \rangle$ in 
$(q;q)_\infty^{\chi(X)} \cdot F_2^{\alpha}(q)$ 
in the next four lemmas respectively.

\begin{lemma}   \label{eXalpha} 
The coefficient of $\langle e_X, \alpha \rangle$ in 
$(q;q)_\infty^{\chi(X)} \cdot F_2^{\alpha}(q)$ is equal to
\begin{eqnarray}   \label{lemma_ch2L.6}
-\frac{7}{24} Z(4) - \frac{23}{24} Z(2)^2.
\end{eqnarray}
\end{lemma}
\begin{proof}
By Proposition~\ref{prop4thterms}, 
the coefficient of $\langle e_X, \alpha \rangle$ in 
$(q;q)_\infty^{\chi(X)} \cdot F_2^{\alpha}(q)$ equals
\begin{eqnarray*}   
& &-\sum_{n > 0} \frac{n^2 - 1}{12} \frac{q^n}{(1-q^n)^2}
   - \sum_{i, j > 0} \frac{iq^i}{1-q^i} \frac{q^j}{(1-q^j)^2}  \\
&=&\frac{1}{12} Z(2) - \frac{1}{12} \sum_{n > 0} \frac{n^2 q^n}{(1-q^n)^2}
   - \sum_{i> 0} \frac{iq^i}{1-q^i} \cdot 
   \sum_{j > 0} \frac{q^j}{(1-q^j)^2} \\
&=&\frac{1}{12} Z(2) - \frac{1}{12} \sum_{n > 0} \frac{n^2 q^n}{(1-q^n)^2}
   - \sum_{i> 0} \frac{iq^i}{1-q^i} \cdot Z(2)  \\
&=&-\frac{7}{24} Z(4) - \frac{23}{24} Z(2)^2
\end{eqnarray*}
where we have used Lemma~\ref{lemma_n2qn} in the last step.
\end{proof}

\begin{lemma}   \label{1Xalpha} 
The coefficient of $\langle 1_X, \alpha \rangle$ in 
$(q;q)_\infty^{\chi(X)} \cdot F_2^{\alpha}(q)$ is equal to
$$
\frac1{12} Z(4) - \frac13 Z(2)^2.
$$
\end{lemma}
\begin{proof}
By Proposition~\ref{prop4thterms}, 
the coefficient of $\langle 1_X, \alpha \rangle$ in 
$(q;q)_\infty^{\chi(X)} \cdot F_2^{\alpha}(q)$ equals
\begin{eqnarray*}     
\sum_{\substack{\la = (-i,-j,-k, i+j+k)\\i \ge j \ge k > 0}} 
    \frac{1}{\la^!} \frac{q^{i+j+k}}{(1-q^{i+j+k})(1-q^i)(1-q^j)(1-q^k)}
\end{eqnarray*}
\begin{eqnarray*}     
+ \sum_{\substack{\la = (-i-j-k, k, j, i)\\i \ge j \ge k > 0}} 
    \frac{1}{\la^!} \frac{q^{i+j+k}}{(1-q^i)(1-q^j)(1-q^k)(1-q^{i+j+k})}
\end{eqnarray*}
\begin{eqnarray*}     
- \sum_{\substack{\la = (-i, -j, \ell, k)\\i \ge j > 0, 
      k \ge \ell > 0, i+j=k+\ell}} 
   \frac{1}{\la^!} \frac{q^{i+j}}{(1-q^i)(1-q^j)(1-q^k)(1-q^{\ell})}.
\end{eqnarray*}
Note that 
\begin{eqnarray}    \label{1Xalpha.1} 
& &\sum_{\substack{\la = (-i,-j,-k, i+j+k)\\i \ge j \ge k > 0}} 
   \frac{1}{\la^!} \frac{q^{i+j+k}}{(1-q^{i+j+k})(1-q^i)(1-q^j)(1-q^k)}
   \nonumber  \\
&=&\sum_{\substack{\la = (-i-j-k, k, j, i)\\i \ge j \ge k > 0}} 
    \frac{1}{\la^!} \frac{q^{i+j+k}}{(1-q^i)(1-q^j)(1-q^k)(1-q^{i+j+k})}
   \nonumber  \\
&=&\frac16 \sum_{i,j,k > 0} 
   \frac{q^{i+j+k}}{(1-q^i)(1-q^j)(1-q^k)(1-q^{i+j+k})}.
\end{eqnarray}
In addition, we see that 
\begin{eqnarray}    \label{1Xalpha.2} 
& &\sum_{\substack{\la = (-i, -j, \ell, k)\\i \ge j > 0, 
      k \ge \ell > 0, i+j=k+\ell}} 
   \frac{1}{\la^!} \frac{q^{i+j}}{(1-q^i)(1-q^j)(1-q^k)(1-q^{\ell})}
   \nonumber  \\
&=&\frac14 \, \sum_{i,j,k, \ell >0, i+j=k+\ell}   
   \frac{q^{i+j}}{(1-q^i)(1-q^j)(1-q^k)(1-q^{\ell})}.
\end{eqnarray}
Therefore, the coefficient of $\langle 1_X, \alpha \rangle$ in 
$(q;q)_\infty^{\chi(X)} \cdot F_2^{\alpha}(q)$ is equal to
\begin{eqnarray*}
& &\frac13 \sum_{i,j,k > 0} 
   \frac{q^{i+j+k}}{(1-q^i)(1-q^j)(1-q^k)(1-q^{i+j+k})}   \\
& &- \frac14 \, \sum_{i,j,k, \ell >0, i+j=k+\ell}   
   \frac{q^{i+j}}{(1-q^i)(1-q^j)(1-q^k)(1-q^{\ell})}   \\
&=&-\Theta_2(q)  \\
&=&\frac1{12} Z(4) - \frac13 Z(2)^2
\end{eqnarray*}
where we have applied \eqref{i+j=k+l.1} and \eqref{i+j=k+l.4}.
\end{proof}

\begin{lemma}   \label{KXalpha} 
The coefficient of $\langle K_X, \alpha \rangle$ in 
$(q;q)_\infty^{\chi(X)} \cdot F_2^{\alpha}(q)$ is equal to 
$$
-\frac1{6} Z(4) + \frac23 Z(2)^2.
$$
\end{lemma}
\noindent
{\it Proof.}
By Proposition~\ref{prop4thterms}, 
the coefficient of $\langle K_X, \alpha \rangle$ in 
$(q;q)_\infty^{\chi(X)} \cdot F_2^{\alpha}(q)$ equals
$$
-\sum_{\substack{\la = (-i,-j,-k, i+j+k)\\i \ge j \ge k > 0}} 
    \frac{1}{\la^!} \frac{q^{i+j+k}}{(1-q^{i+j+k})(1-q^i)(1-q^j)(1-q^k)}
$$
$$
- 3\sum_{\substack{\la = (-i-j-k, k, j, i)\\i \ge j \ge k > 0}} 
    \frac{1}{\la^!} \frac{q^{i+j+k}}{(1-q^i)(1-q^j)(1-q^k)(1-q^{i+j+k})}
$$
$$
+ 2\sum_{\substack{\la = (-i, -j, \ell, k)\\i \ge j > 0, 
      k \ge \ell > 0, i+j=k+\ell}} 
   \frac{1}{\la^!} \frac{q^{i+j}}{(1-q^i)(1-q^j)(1-q^k)(1-q^{\ell})}.
$$
By \eqref{1Xalpha.1} and \eqref{1Xalpha.2}, 
the coefficient of $\langle K_X, \alpha \rangle$ in 
$(q;q)_\infty^{\chi(X)} \cdot F_2^{\alpha}(q)$ is equal to
$$
-\frac23 \sum_{i,j,k > 0} 
   \frac{q^{i+j+k}}{(1-q^i)(1-q^j)(1-q^k)(1-q^{i+j+k})}  
$$
$$
+ \frac12 \, \sum_{i,j,k, \ell >0, i+j=k+\ell}   
   \frac{q^{i+j}}{(1-q^i)(1-q^j)(1-q^k)(1-q^{\ell})}  
$$
By \eqref{i+j=k+l.1} and \eqref{i+j=k+l.4}, 
the coefficient of $\langle K_X, \alpha \rangle$ in 
$(q;q)_\infty^{\chi(X)} \cdot F_2^{\alpha}(q)$ is equal to
\begin{equation}
2 \Theta_2(q) = -\frac1{6} Z(4) + \frac23 Z(2)^2.
\tag*{$\qed$}
\end{equation}

\begin{lemma}   \label{KX2alpha} 
The coefficient of $\langle K_X^2, \alpha \rangle$ in 
$(q;q)_\infty^{\chi(X)} \cdot F_2^{\alpha}(q)$ is equal to
\begin{eqnarray}   \label{lemma_ch2L.7}
\frac{13}{12} Z(4) - \frac13 Z(2)^2 - \frac14 Z(3) + \frac14 Z(2).
\end{eqnarray}
\end{lemma}
\begin{proof}
By Proposition~\ref{prop4thterms}, 
the coefficient of $\langle K_X^2, \alpha \rangle$ in 
$(q;q)_\infty^{\chi(X)} \cdot F_2^{\alpha}(q)$ equals
\begin{eqnarray*}   
& &\sum_{n > 0} \frac{(n-1)(2n-1)}{12} \frac{q^n}{(1-q^n)^2} 
   + \frac12 \sum_{i, j > 0} \frac{i+j-1}{2}
   \frac{q^{i+j}}{(1-q^i)(1-q^j)(1-q^{i+j})}   \\
& &+ 3 \sum_{\substack{\la = (-i-j-k, k, j, i)\\i \ge j \ge k \ge 1}}  
   \frac{1}{\la^!} \frac{q^{i+j+k}}{(1-q^i)(1-q^j)(1-q^k)(1-q^{i+j+k})} \\
& &- \sum_{\substack{\la = (-i, -j, \ell, k)\\i \ge j \ge 1, 
   k \ge \ell \ge 1, i+j=k+\ell}}   
   \frac{1}{\la^!} \frac{q^{i+j}}{(1-q^i)(1-q^j)(1-q^k)(1-q^{\ell})} \\
&=&\frac{1}{12}\sum_{n > 0} \frac{(n-1)(2n-1)q^n}{(1-q^n)^2} 
   + \frac{1}{4}\sum_{i, j > 0} 
   \frac{(i+j-1)q^{i+j}}{(1-q^i)(1-q^j)(1-q^{i+j})}  \\
& &+ \frac12 \sum_{i,j,k > 0} 
   \frac{q^{i+j+k}}{(1-q^i)(1-q^j)(1-q^k)(1-q^{i+j+k})} \\
& &- \frac14 \sum_{i,j,k, \ell >0, i+j=k+\ell}   
   \frac{q^{i+j}}{(1-q^i)(1-q^j)(1-q^k)(1-q^{\ell})}
\end{eqnarray*}
where we have applied \eqref{1Xalpha.1} and \eqref{1Xalpha.2}.
By Lemma~\ref{lemma_n2qn}, we have
\begin{eqnarray*}   
& &\frac{1}{12}\sum_{n > 0} \frac{(n-1)(2n-1)q^n}{(1-q^n)^2} \\
&=&\frac{1}{6} \left (\frac72 Z(4) - \frac12 Z(2)^2 + Z(2) \right ) 
   - \frac{1}{4} \sum_{n > 0} \frac{nq^n}{(1-q^n)^2} +\frac{1}{12} Z(2) \\
&=&\frac7{12} Z(4) - \frac1{12} Z(2)^2 + \frac14 Z(2) 
   - \frac{1}{4} q\frac{\rm d}{{\rm d}q}[1].
\end{eqnarray*}
By Lemma~\ref {ijkCO}, $\displaystyle{\frac{1}{4}\sum_{i, j > 0} 
   \frac{(i+j-1)q^{i+j}}{(1-q^i)(1-q^j)(1-q^{i+j})}}$ is equal to 
$$
\frac{1}{2}\sum_{n > m > 0} \frac{nq^n}{(1-q^n)^2(1-q^m)}
   + \frac{1}{2} q\frac{\rm d}{{\rm d}q}[1]
- \frac78 Z(4) + \frac18 Z(2)^2 - \frac14 Z(3)- \frac14 Z(2).
$$
Together with Lemma~\ref{i+j=k+l}, the coefficient of 
$\langle K_X^2, \alpha \rangle$ in 
$(q;q)_\infty^{\chi(X)} \cdot F_2^\alpha(q)$ equals
\begin{eqnarray*}    
& &-\frac{5}{24} Z(4) - \frac7{24} Z(2)^2 - \frac14 Z(3)
   + \frac{1}{2}\sum_{n > m > 0} \frac{nq^n}{(1-q^n)^2(1-q^m)}   \nonumber  \\
& &+ \frac{1}{4} q\frac{\rm d}{{\rm d}q}[1]
   + \frac16 \sum_{i,j,k > 0} 
   \frac{q^{i+j+k}}{(1-q^i)(1-q^j)(1-q^k)(1-q^{i+j+k})}   \\
&=&\frac{13}{12} Z(4) - \frac13 Z(2)^2 - \frac14 Z(3) + \frac14 Z(2)
\end{eqnarray*}
where we have used Lemma~\ref{lemma_ijk100} in the last step. 
\end{proof}

The following is the main result in this section. 

\begin{theorem} \label{lemma_ch2L}
Let $\alpha \in H^*(X)$ where $X$ is a smooth projective surface. Then, 
the reduced series $(q;q)_\infty^{\chi(X)} \cdot F_2^\alpha(q)$ 
is equal to 
$$ 
\left (-\frac{7}{24} Z(4) - \frac{23}{24} Z(2)^2 \right )
\langle e_X, \alpha \rangle
+ \left (\frac1{12} Z(4) - \frac13 Z(2)^2 \right ) 
\langle 1_X, \alpha \rangle
$$ 
$$
+ \left (-\frac1{6} Z(4) + \frac23 Z(2)^2 \right ) 
\langle K_X, \alpha \rangle
+ \left (\frac{13}{12} Z(4) - \frac13 Z(2)^2 - \frac14 Z(3) 
+ \frac14 Z(2) \right ) \langle K_X^2, \alpha \rangle.
$$
\end{theorem}
\begin{proof}
Follows from Proposition~\ref{prop4thterms} and the four 
Lemmas~\ref{eXalpha}-\ref{KX2alpha}.
\end{proof}

\section{\bf Application to Okounkov's Conjecture} 
\label{sect_Application}

Let $L$ be a line bundle on the smooth projective surface $X$. 
Recall from \eqref{L[n]} the tautological rank-$n$ bundle $\Ln$ 
over the Hilbert scheme $\Xn$.
By the Grothendieck-Riemann-Roch Theorem and \eqref{DefOfGGammaN}, we obtain
\begin{eqnarray}     \label{chLnGRR}
   \ch (\Ln) 
&=&p_{1*}(\ch({\mathcal O}_{{\mathcal Z}_n}) \cdot p_2^*\ch(L) \cdot p_2^*{\rm td}(X) )
               \nonumber    \\
&=&p_{1*}(\ch({\mathcal O}_{{\mathcal Z}_n}) \cdot p_2^*(1_X+L+L^2/2) \cdot p_2^*{\rm td}(X) )
               \nonumber    \\
&=&G(1_X, n) + G(L, n) + G(L^2/2, n).
\end{eqnarray}
Since the cohomology degree of $G_i(\alpha, n)$ is $2i+ |\alpha|$, we have
\begin{eqnarray}     \label{chLnGAlpha}
\ch_k (\Ln) = G_k(1_X, n) + G_{k-1}(L, n) + G_{k-2}(L^2/2, n).
\end{eqnarray} 

Following \cite{Oko}, we have defined the generating series 
$\big \langle \ch_{k_1}^{L_1} \cdots \ch_{k_N}^{L_N} \big \rangle$ 
and its reduced version $\big \langle \ch_{k_1}^{L_1} \cdots 
\ch_{k_N}^{L_N} \big \rangle'$
in \eqref{OkoChkN.1} and \eqref{OkoChkN.2} respectively. 

\begin{example}    \label{203001031038uuu}
Let $L$ be a line bundle over a surface $X$. 
By \eqref{OkoChkN.2}, \eqref{chLnGAlpha} and \eqref{F-generating}, 
$$
\big \langle \ch_{1}^L \big \rangle'
= (q; q)_\infty^{\chi(X)} \cdot \big \langle \ch_{1}^L \big \rangle
= (q; q)_\infty^{\chi(X)} \cdot \big (F_1^{1_X}(q) + F_0^L(q) \big )
$$
where by abusing notations, we have also used $L$ to denote 
its first Chern class. By \cite[Proposition~5.15]{Qin1}, 
for $\alpha \in H^*(X)$, we have
\begin{eqnarray}  \label{203001031038uuu.3}
(q; q)_\infty^{\chi(X)} \cdot F_0^\alpha(q) 
= \langle 1_X - K_X + e_X, \alpha \rangle \cdot Z(2).
\end{eqnarray}
By \cite[Proposition~5.17]{Qin1} (the assumption $e_X \alpha = 0$ 
there can be dropped), 
\begin{eqnarray}  \label{203001031038uuu.2}
& &(q; q)_\infty^{\chi(X)} \cdot F_1^{\alpha}(q)   \nonumber  \\
&=&\frac{\langle K_X - K_X^2, \alpha \rangle}{2} \cdot 
  \left (\sum_{m > 0} \frac{(m-1) q^m}{(1-q^m)^2}
  + \sum_{m_1, m_2 > 0} \frac{q^{m_1}}{1-q^{m_1}} \frac{q^{m_2}}{1-q^{m_2}} 
  \frac{1}{1-q^{m_1+m_2}} \right )   \nonumber  \\
&=&\frac{\langle K_X - K_X^2, \alpha \rangle}{2} \cdot 
  \big (-Z(2) + Z(3) \big ).
\end{eqnarray}
Therefore, we obtain
$(q; q)_\infty^{\chi(X)} \cdot F_0^L(q) = -Z(2) \cdot K_X L$, 
and $(q; q)_\infty^{\chi(X)} \cdot F_1^{1_X}(q)
= -{K_X^2}/{2} \cdot \big (-Z(2) + Z(3) \big )$. Hence we have
\begin{eqnarray}   \label{203001031038uuu.1}
\big \langle \ch_{1}^L \big \rangle'
= \frac{1}{2} \big (Z(2) - Z(3) \big ) \cdot K_X^2
  - Z(2) \cdot K_X L.
\end{eqnarray}
When $L = \mathcal O_X$, \eqref{203001031038uuu.1} in 
the equivariant setting is given by \cite[Corollary~3]{CO}.
\end{example}

Next, we study $\big \langle \ch_2^L \big \rangle'$. 
By \eqref{chLnGAlpha} and \eqref{F-generating}, 
$\big \langle \ch_2^L \big \rangle'$ is equal to
\begin{eqnarray}     \label{chLnGAlpha.1}
& &(q; q)_\infty^{\chi(X)} \cdot \big \langle \ch_2^L \big \rangle
      \nonumber   \\
&=&(q; q)_\infty^{\chi(X)} \cdot \sum_{n \ge 0} q^n \, 
      \int_\Xn \ch_2(L^{[n]}) \cdot c(T_\Xn)    \nonumber   \\
&=&(q; q)_\infty^{\chi(X)} \cdot F_2^{1_X}(q)
      + (q; q)_\infty^{\chi(X)} \cdot F_1^L(q)
      + (q; q)_\infty^{\chi(X)} \cdot F_0^{L^2/2}(q).
\end{eqnarray}
By \eqref{203001031038uuu.3} and \eqref{203001031038uuu.2}, we obtain
\begin{eqnarray}   \label{ch2L.0}
  \langle \ch_2^L \rangle'
= (q;q)_\infty^{\chi(X)} \cdot F_2^{1_X}(q) 
  + \frac12 (Z(3) - Z(2)) \cdot K_X L + \frac12 Z(2)\cdot L^2.
\end{eqnarray}

\begin{theorem} \label{theorem_ch2L}
Let $L$ be a line bundle over a smooth projective surface $X$, 
and $K_X$ be the canonical divisor of $X$. Then, 
the reduced series $\langle \ch_2^L \rangle'$ is equal to
$$ 
\left (-\frac{7}{24} Z(4) - \frac{23}{24} Z(2)^2 \right ) \chi(X)
+ \frac12 (Z(3) - Z(2)) \cdot K_X L 
$$
$$ 
+ \left (\frac{13}{12} Z(4) - \frac13 Z(2)^2 - \frac14 Z(3) 
+ \frac14 Z(2) \right ) K_X^2 + \frac12 Z(2) \cdot L^2.
$$
In particular, Conjecture~\ref{OkoConj} holds for the reduced series 
$\langle \ch_2^L \rangle'$.
\end{theorem}
\begin{proof}
By Theorem~\ref{lemma_ch2L}, $(q;q)_\infty^{\chi(X)} \cdot F_2^{1_X}(q)$ 
is equal to 
$$ 
\left (-\frac{7}{24} Z(4) - \frac{23}{24} Z(2)^2 \right ) \chi(X)
+ \left (\frac{13}{12} Z(4) - \frac13 Z(2)^2 - \frac14 Z(3) 
+ \frac14 Z(2) \right ) K_X^2.
$$
Combining with \eqref{ch2L.0} completes the proof of the theorem.
\end{proof}

\begin{corollary} \label{corollary_ch2L}
Let $L$ be a line bundle over a smooth projective surface $X$. 
If the canonical divisor of $X$ is numerically trivial, 
then $\langle \ch_2^L \rangle'$ 
is a quasi-modular form of weight at most $4$. 
In particular, Conjecture~\ref{QinConj} holds for 
$\langle \ch_2^L \rangle'$.
\end{corollary}
\begin{proof}
Since $K_X$ is numerically trivial, 
we see from Theorem~\ref{theorem_ch2L} that
$$ 
  \langle \ch_2^L \rangle'
= \left (-\frac{7}{24} Z(4) - \frac{23}{24} Z(2)^2 \right ) \chi(X)
  + \frac12 Z(2) \cdot L^2.
$$
Note that $\chi(X), L^2 \in \Z$. 
Since the set of all quasi-modular forms is the algebra 
$\Q\big [Z(2), Z(4), Z(6) \big ]$, 
$\langle \ch_2^L \rangle'$ is a quasi-modular form of weight at most $4$.
\end{proof}


\begin{thebibliography}{AB}

\bibitem{Bac} H. Bachmann, 
{\em The algebra of bi-brackets and regularized multiple Eisenstein 
series.} J. Number Theory {\bf 200} (2019), 260-294. 

\bibitem{BK1} H. Bachmann, U. K\" uhn,
{\em The algebra of generating functions for multiple divisor sums
and applications to multiple zeta values}. 
Rmanujan J. {\bf 40} (2016), 605-648.

\bibitem{BK2} H. Bachmann, U. K\" uhn,
{\em A dimension conjecture for $q$-analogues of multiple zeta values}. 
Periods in Quantum Field Theory and Arithmetic, Springer Proceedings 
in Mathematics \& Statistics {\bf 314} (2020), 237-258.

\bibitem{BK3} H. Bachmann, U. K\" uhn,
{\em A short note on a conjecture of Okounkov
about a $q$-analogue of multiple zeta values}. arXiv:1407.6796.


\bibitem{Bra1} D. Bradley,
{\em Multiple $q$-zeta values}. J. Algebra {\bf 283} (2005), 752-798.

\bibitem{Bra2} D. Bradley,
{\em On the sum formula for multiple $q$-zeta values}. 
Rocky Mountain J. Math. {\bf 37} (2007), 1427-1434.

\bibitem{Bri} J. Briancon, 
{\em Description de ${\rm Hilb}^n \mathbb C \{x, y\}$}. 
Invent. Math. {\bf 41} (1977), 45-89.

\bibitem{Car1} E. Carlsson, 
{\em Vertex operators and moduli spaces of sheaves}. 
Ph.D thesis, Princeton University, 2008. arXiv:0906.1825v1

\bibitem{Car2} E. Carlsson, 
{\em Vertex operators and quasimodularity of Chern numbers on the Hilbert scheme}. 
Adv. Math. {\bf 229} (2012), 2888-2907.

\bibitem{CO} E. Carlsson, A. Okounkov,
{\em Exts and Vertex Operators}. Duke Math. J. {\bf 161} (2012), 1797-1815.

\bibitem{Fog} J. Fogarty,
{\em Algebraic families on an algebraic surface}. Amer. J. Math.
{\bf 90} (1968), 511-521.

\bibitem{Got} L. G\"ottsche,
{\em The Betti numbers of the Hilbert scheme of points on a smooth
projective surface}. Math. Ann. {\bf 286} (1990) 193--207.

\bibitem{Gro} I.~Grojnowski,
{\em Instantons and affine algebras I: the Hilbert scheme and
vertex operators}, Math. Res. Lett. {\bf 3} (1996) 275--291.

\bibitem{Grot} A. Grothendieck,
{\em Techniques de construction et th{\' e}or{\`e}mes 
d'existence en g\' eom\'etrie alg\' ebrique IV: 
Les sch\'emas de Hilbert}.
Sem. Bourbaki {\bf 221}, 13. 1960-1961.

\bibitem{Iar} A. Iarrobino, 
{\em Punctual Hilbert schemes}.
Bull. Amer. Math. Soc. {\bf 78} (1972), 819-823.


\bibitem{LQW1} W.-P. Li, Z. Qin and W. Wang, 
{\em Vertex algebras and the
cohomology ring structure of Hilbert schemes of points on
surfaces}. Math. Ann. {\bf 324} (2002), 105-133.

\bibitem{LQW2} W.-P. Li, Z. Qin and W. Wang, 
{\em Hilbert schemes and $\mathcal W$ algebras}.
Intern. Math. Res. Notices {\bf 27} (2002), 1427-1456.


\bibitem{Nak} H. Nakajima,
{\em Heisenberg algebra and Hilbert schemes of points on
projective surfaces}, Ann. Math. {\bf 145} (1997) 379--388.

\bibitem{Oko} A. Okounkov,
{\em Hilbert schemes and multiple $q$-zeta values}. 
Funct. Anal. Appl. {\bf 48} (2014), 138-144.

\bibitem{OT} J. Okuda, Y. Takeyama,
{\em On relations for the multple $q$-zeta values}. 
Ramanujan J. {\bf 14} (2007), 379-387.

\bibitem{Qin1} Z. Qin, 
{\em Hilbert schemes of points and infinite dimensional Lie algebras}.
Mathematical Surveys and Monographs {\bf 228},
American Mathematical Society, Providence, RI, 2018.

\bibitem{Qin2} Z. Qin, 
{\em A quick survey from S-duality conjecture of Vafa-Witten to 
a conjecture of Okounkov}. Presentation at the SQuaRE Workshop 
``Moduli of sheaves on surfaces via Bridgeland stability", 
American Institute of Mathematics, San Jose, CA, 2022.

\bibitem{QY} Z. Qin, F. Yu, 
{\em On Okounkov's conjecture connecting Hilbert schemes of points 
and multiple $q$-zeta values}. Intern. Math. Res. Notices {\bf 2018}, 321-361.

\bibitem{SQ} Z. Shen, Z. Qin, 
{\em Hilbert schemes of points and quasi-modularity}. 
Pure Appl. Math. Q. {\bf 16} (2020), 1697-1730. 

\bibitem{Tang} T. Tang, 
{\em Higher order derivatives of Heisenberg operators on
the cohomology of the Hilbert schemes of points}. 
Undergraduate Summer Research Internship. Hong Kong University of 
Science and Technology and University of Missouri, 2019.

\bibitem{Zhao} J. Zhao,
{\em Uniform approach to double shuffle and duality relations of 
various $q$-analogs of multiple zeta values via Rota-Baxter algebras}.
Preprint. arXiv:1412.8044  

\bibitem{Zhou} Jian Zhou, 
{\em On quasimodularity of some equivariant intersection numbers 
on the Hilbert schemes}. Preprint. arXiv:1801.09090

\bibitem{Zud} W. Zudilin,
{\em Algebraic relations for multiple zeta values}. 
Russian Math. Surveys {\bf 58} (2003), 1-29.

\end{thebibliography}
\end{document}